\documentclass[10pt]{amsart}
\usepackage{amscd}

\newcommand{\bT}{{\mathbb T}}
\newcommand{\bZ}{{\mathbb Z}}
\newcommand{\bF}{{\mathbb F}}

\newcommand{\bC}{{\mathbb C}}

\newcommand{\bG}{{\mathbb G}}

\newtheorem{thm}{Theorem}[section]
\newtheorem{lemma}[thm]{Lemma}

\newtheorem{cor}[thm]{Corollary}
\newtheorem{prop}[thm]{Proposition}

\numberwithin{equation}{section}

\newtheorem{conj}[thm]{Conjecture}
\newcommand{\II}{{I_{\infty}^2}}

\title[   $BG$ and $G/B$]{Notes on Chow rings of 
 $G/B$ and $BG$}
\author[N.Yagita]{Nobuaki Yagita}

\begin{document}

\address{[N. Yagita]
Department of Mathematics, Faculty of Education, Ibaraki University,
Mito, Ibaraki, Japan}
 
\email{nobuaki.yagita.math@vc.ibaraki.ac.jp}
\keywords{Chow rings, flag varieties,  classifying spaces}
\subjclass[2010]{Primary 55P35, 57T25, 20C20 ; Seconary 
 55R35, 57T05}

\maketitle

\begin{abstract}
Let $G$ be a compact Lie group and $T$ its maximal torus.  The composition of maps
$ H^*(BG)\to H^*(BT) \to H^*(G/T)$
is zero for positive degree, while it is far from exact.
We change $H^*(G/T)$ by Chow ring  $CH^*(X)$ for
$X$ some twisted form of $G/T$, and change
$H^*(BG)$ by $CH^*(BG)$.  Then we see that it becomes
near to exact but still not exact, in general.  We also 
see that
the difference for exactness relates
 to the generalized Rost motive in $X$.
\end{abstract}

\section{Introduction}

Let $p$ be a prime number. 
Let $G$ and  $T$ be a  connected compact Lie group
and its maximal torus.
Given a field $k$ with $ch(k)=0$,
let $G_k$ and $T_k$ be a split  reductive group and 
a split maximal torus over the field $k$,
corresponding to $G$ and  $T$.  Let $B_k$ be the Borel
subgroup containing $T_k$.
Let us write by $BG_k$ its classifying space of $G_k$
defined by Totaro \cite{To1}.

For a smooth algebraic variety $X$ over $k$   (resp. toplogical space),  let
$CH^*(X)=CH^*(X)_{(p)}$ (resp. $H^*(X)=H^*(X)_{(p)}$) mean $p$-localized Chow ring over $k$  (resp. $p$-localized ordinaly cohomology ring).  In general, to compute $CH^*(BG_k)$
or $H^*(BG)$ are difficult problems.  At first,  we consider
them modulo torsion elements. 
We consider the following diagram.
\[ \begin{CD}
(1.1)\qquad   CH^*(BG_k)/Tor  \qquad    @>{(1)}>{i^*}> CH^*(BB_k)^W \\
    @V{(2)}V{cl}V \quad             @V{\cong}VV  \\
   H^*(BG)/Tor  @>{(3)}>{i^*}> H^*(BT)^W \end{CD} \]
where $Tor$ is the ideal generated by torsion elements, $cl$ is the cycle map,  and $W=N_G(T)/T$
is the Weyl group.  

When $H^*(G)$ is torsion free,
we know that $Tor=0$ and all maps $(1),(2),(3)$ are isomorphic.  
Hence we only consider cases that $H^*(G)$ have 
$p$-torsion throughout this paper. 
By the existence of the Becker-Gottlieb transfer,
the maps $(1),(3)$ are injections.  Moreover when $G$ is simply connected,  $(1)$ is always not surjective
(\cite{YaW}), while for many cases $(3)$ are  surjective.
(For cases that $(3)$ are not surjective are founded
by Feshbach \cite{Fe}, Benson-Wood \cite{Be-Wo}). In any way,
$CH^*(BG_k)/Tor$ is isomorphic to a proper subring of
$CH^*(BB_k)^W$
for simply connected  $G$.

To study $CH^*(BG_k)/Tor$, we consider twisted
flag varieties.
Let $\bG$ be a $G_k$-torsor. 
 Then $\bF=\bG/B_k$ is a
(twisted) form of the flag variety $G_k/B_k$. 
The fibering $G/T\stackrel{j}{\to} BT 
\stackrel{i}{\to} BG$ induces
the maps 
\[ (1.2)\quad  CH^*(BG_k)\stackrel{i^*}{\to} CH^*(BB_k) 
\stackrel{j^*}{\to}  CH^*(\bF), \]
whose composition $j^*i^*=0$  for $*>0$.  But it 
 is far from exact  when $\bG\cong G_k$ the split group.
  Here exact means
$Ker(j^+)=Ideal(Im(i^+))\subset CH^+(BB_k)$.  
However, we observe that 
it becomes near exact when $\bG$ is sufficient
twisted, while it is still not exact for most cases.
To see this fact, we define the difference as
\[(1.3)\quad D_{CH}(\bG)=Ker(j^+(\bG))/(Ideal(Im(i^+)),\]
and will see that $D_{CH}(\bG))$ is quite small.
for versal  flag variety $CH^*(\bF)$.
(For the definition of $versal$ $torsor$ see 
[Ga-Me-Se], [ Me-Ne-Za], [Kar], [To2].)
In particular, when $G$ is of type $(I)$,
a non-trivial  $G_k$-torsor is always versal.

By Petrov-Semenov-Zainoulline ([Pe-Se-Za], [Se-Zh]), it is known that the $p$-localized motive
$M(\bF)_{(p)}$ of $\bF$ is decomposed as
\[ (1.4)\quad M(\bF)_{(p)}=M(\bG/B_k)_{(p)}\cong  R(\bG)\otimes(\oplus_i  \bT^{\otimes s_i})\]
where $\bT$ is the (reduced) Tate motive and $R(\bG)$ is  some
motive called generalized Rost motive. (It is the original 
Rost motive ([Ro], [Vo1,2], [Pe-Se-Za], \cite{YaC}) when $G$ is of type $(I)$ as explained below).
Hence we have maps
\[ (1.5)\quad CH^*(BB_k)\stackrel{j^*}{\to} CH^*(\bF)
\stackrel{pr.}{\to} 
CH^*(R(\bG)).\]
From Merkurjev and Karpenko [Me-Ne-Za], [Kar], we know
that the first map $j^*$ is also surjective when $\bG$ is a versal
$G_k$-torsor. 

For ease of computations, we mainly consider the $mod(p)$
theories  for (1.2)
\[  (1.6)\quad CH^*(BG_k)/p\stackrel{i_p^*}{\to} CH^*(BB_k)/p 
\stackrel{j_p^*}{\to}  CH^*(\bF)/p.\]
Let us define $D(\bG)=Ker(j_p^+)/(Ideal(Im(i_p^+)).$ 
 Then we see
\begin{lemma}
Let $\bG$ be versal.
 Then we have the surjection
\[  pr: D(G_k)/D(\bG) \to CH^+(R(\bG))/p.\]
\end{lemma}

We will see that  $D(\bG)$ are quite small
in some cases.  For example we have 
\begin{thm}
Let $(G,p)=(SO(2\ell+1), 2)$ and $\bG$ be versal.
Then $ D(\bG)\cong 0$, that is the above  sequence 
(1.6) is exact.
\end{thm}
\begin{thm}  Let $(G(N),p)=(Spin(N),2)$ and 
$\bG(N)$ be versal.  Then we have
$  lim_{\infty \gets N}D(\bG(N))=0. $
\end{thm}\begin{prop}  Let $G=Spin(7)$ and $\bG$ be versal.
Then  we have additively
\[ D(\bG)\cong \Lambda(c_2c_3,e_4)^+\otimes
S(t,c)\quad for\ S(t,c)=S(t)/(c_2,c_3,e_4)\]
where $c_i$ is the $i$-th elementary symmetric function  in $S(t)=CH^*(BB_k)$ and $e=c_1^4$,
and $\Lambda(a,...,b)$ is the $\bZ/2$-exterior algebra generated by $a,...,b$ 
\end{prop}

The plan of this paper is the following.
In $\S 2$, we recall the  Chow ring $CH^*(\bF)$
for a nontrivial $G_k$-torsor $\bG$.  In $\S 3$
we note some elementary relations between
$CH^*(\bF)$ and  $CH^*(BG_k)$. In $\S 4$
we note some facts for $CH^*(B_k)^W/Tor$.
In $\S 5, \S 6$, we try to compute
$D(\bG)$ for $G=PU(p), SO(n)$.
In $\S7, \S 8$, we try to study $D(\bG)$ for
$G=Spin(n) $ for general $n$.  In  $\S 9,\S 10$, we study
$Spin(7),Spin(9)$.  In $\S 11, \S12$ we study the case
$(G,p)=(F_4,3)$. In $\S 13$, we study the case
$G=E_6,E_7$ and $p=3$.

The author thanks Akihiko Hida and Masaki Kameko
for suggestions for this paper.  In particular Masaki Kameko found errors  in the first version of this paper.

\section{$CH^*(\bG/B_k)$}

We recall arguments for
$H^*(G/T)$ in algebraic topology. 
By Borel, its $mod(p)$ cohomology is (for $p$ odd)
\[ H^*(G;\bZ/p)\cong P(y)/p\otimes \Lambda(x_1,...,x_{\ell}),
\quad |x_i|=odd\]
where  $P(y)$ is a truncated polynomial ring 
generated by $even$ dimensional elements $y_i$,
and $\Lambda(x_1,...,x_{\ell})$ is the $\bZ/p$-exterior algebra generated by $x_1,...,x_{\ell}$.  When $p=2$, we consider the  graded ring $grH^*(G;\bZ/2)$ which is isomorphic to the right  hand side ring above.

When $G$ is simply connected and $P(y)$ is generated by just one generator, 
we say that $G$ is of type $(I)$.  Except for 
$(E_7,p=2)$ and $(E_8,p=2,3)$, all exceptional (simple) Lie groups are of type $(I)$.  The spin groups $G=Spin(n)$ are
of type $(I)$ for $7\le n\le 10$.
Note that in these cases, it is known
$rank(G)=\ell\ge 2p-2$.

We consider the fibering ([Tod2], [Mi-Ni])
$G\stackrel{\pi}{\to}G/T\stackrel{i}{\to}BT$
and the induced spectral sequence 
\[(2.1)\quad  E_2^{*,*}=H^*(BT;H^*(G;\bZ/p)) \Longrightarrow H^*(G/T;\bZ/p).\] 
Here we  can write 
$H^*(BT)\cong S(t)=\bZ[t_1,...,t_{\ell}]$
with $|t_i|=2.$

It is well
known that $y_i\in P(y)$ are permanent cycles and 
that there is a regular sequence 
$(\bar b_1,...,\bar b_{\ell})$ in $H^*(BT)/(p)$ such that $d_{|x_i|+1}(x_i)=\bar b_i$ ([Tod2], [Mi-Ni]).

We know that $G/T$ is a manifold such that
$H^*(G/T)$ is torsion free and is generated by even degree elements.
We also see that 
there is a filtration in $H^*(G/T)_{(p)}$ such that  
\[ grH^*(G/T)_{(p)}\cong P(y)\otimes S(t)/(b_1,...,b_{\ell})\]
where $b_i\in S(t)$ with $ b_i=\bar  b_i\ mod(p)$.

Recall $BP^*(-)$ be the Brown-Peterson theory with
the coefficient $BP^*=\bZ_{(p)}[v_1,...]$,
$|v_i|=-2(p^i-1)$.  
Then we have 
\[ grBP^*(G/T)\cong BP^*\otimes grH^*(G/T).\] 
Let
$Q_i:H^*(X;\bZ/p)\to H^{*+2p^n-1}(X;\bZ/p)$ be the Milnor operation.  There is a relation between
$Q_i$-action on $H^*(X;\bZ/p)$ and $v_i$-action
on $BP^*(X)$.
\begin{lemma}  
Let $d(x)=b$ in the above spectral sequence (2.1).
Then we can take (a lift of $b$)  $b\in BP^*(G/T)$ such that 
\[ b=\sum_{i=0} v_iy(i) \in \ BP^*(G/T)/\II
\quad with\ I_{\infty}=(p,v_1,...)\]
where $y(i)\in H^*(G/T;\bZ/p)$ with $\pi^*y(i)=Q_ix$.
\end{lemma}

For the algebraic closure $\bar k$ of $k$, let us write $\bar X=X|_{\bar k}$.  Then considering 
(2.1) over $\bar k$, we see
\[CH^*(\bar R(\bG))/p\subset P(y),\quad 
CH^*(\oplus _i  \bT^{\otimes s_i})\cong 
S(t)/(b_1,...,b_{\ell})
.\]

Moreover when $\bG$ is versal, we can see (\cite{YaC}) that
$CH^*(R(\bG))$ is additively generated by
products of $b_1,...,b_{\ell}$ in (2.2) i.e., 
$CH^*(\bar R(\bG)/p\cong P(y)$.
Hence we have surjections 
$CH^*(BB_k)\to CH^*(\bF)\stackrel{pr.}{\to} CH^*(R(\bG)).$

For ease of notations, let us write
\[A(b)=\bZ/p[b_1,...,b_{\ell}],\quad (b)=ideal(b_1,...,b_{\ell})\subset S(t)/p.\]
By giving the filtration on $S(t)$ by $b_i$, we 
can write (additively)
\[gr S(t)/p\cong A(b)\otimes S(t)/(b) .\]
In fact $x\in S(t)/p$ is written as
\[ x=\sum_Ib(I)t(I) \quad  for\
b(I)\in A(b),\ \  and\ \ 0\not =t(I)
\in S(t)/(b).\] 

In particular, we have maps
$ A(b)\stackrel{i_A}{\to} CH^*(\bF)/p\to CH^*(R(\bG))/p.$
We also see that
the above composition map is surjective.
\begin{lemma} (\cite{YaC})
Suppose that there are $f_1(b),...,f_s(b)\in A(b)$ such that 
$CH^*(R(\bG))/p\cong A(b)/(f_1(b),...,f_s(b))$. Moreover if $f_i(b)=0$ for $1\le i\le s$ 
also in $CH^*(\bF)/p$, we have the isomorphism
\[CH^*(\bF)/p\cong S(t)/(p, f_1(b),...,f_s(b)).\]
\end{lemma}

For $N>0$, let us  write \ \ 
$A_N=\bZ/p\{b_{i_1}...b_{i_k}| |b_{i_1}|+...+|b_{i_k}|\le N\}.$
\begin{lemma}
Let $pr: CH^*(\bF)/p\to CH^*(R(\bG))/p$, and 
 $0\not =b\in Ker(pr)$.  Then  
$ b=\sum b'u'$ with $b'\in A_{N},$ $u'\in S(t)^+/(p,b_1,...,b_{\ell})$
 i.e., $ |u'|>0.$
\end{lemma}

Using these, we can prove
\begin{thm} (\cite{YaC})
Let $G$ be  of type $(I)$ and $rank(G)=\ell$.
Let $\bG$ be a non-trivial $G_k$-torsor.
Then $2p-2 \le \ell$, and we can take $b_i\in S(t)=CH^*(BB_k)$ for $1\le i\le \ell$
such that there are 
 isomorphisms
\[ CH^*(R(\bG))/p\cong \bZ/p\{1,b_1,...,b_{2p-2}\},\]
\[ CH^*(X)/p\cong S(t)/(p, b_ib_j,b_k|1\le i,j\le 2p-2<k\le \ell)\]
where  $\bZ/p\{a,b,...\}$ is the $\bZ/p$-free module generated by $a,b,...$
\end{thm}

\section{ relation between $\bG/B_k$ and  $BG$}

Let $h^*(X)=CH^*(X)/I(h)$ for some 
ideal $I(h)$ (e.g., $CH^*(X)/p$).
We note here the following lemma for
 each $G_k$-torsor $\bG$ (not assumed twisted).
\begin{lemma}    For the above $h^*(X)$, 
the composition of the following maps is  zero
for $*>0$
 \[ h^*(BG_k)\to h^*(BB_k)\to h^*(\bG/B_k).\]
\end{lemma}
\begin{proof} Take $U$ (e.g.,  $GL_N$ for a large $N$)
such that $U/G_k$ approximates the classifying space
$BG_k$ [To3].  Namely, we can take $\bG=f^{*}U$ for the classifying map
$f: \bG/G_k \to U/G_k$.
Hence we have  the following commutative diagram
\[  \begin{CD}
  \bF= \bG/B_k @>>> U/B_k\\
       @VVV      @VVV\\
    Spec(k)\cong \bG/G_k @>>> U/G_k
\end{CD}\]  
where $U/B_k$ (resp. $U/G_k$)
approximates $BB_k$ (resp. $BG_k$).
Since $h^*(Spec(k))=CH^*(Spec(k))/I(h)=0$ for $*>0$, we have the lemma.
\end{proof}

The above sequences of maps in the lemma is not
exact, in general.  However we get some informations
from $h^*(\bF)$ to $h^*(BG_k)$. in particular,   we get
much informations of $h^*(BG_k)$
from $h^*(\bF)$ than from $h^*(G_k/B_k)$ when $\bG$ is versal.

Let us write the induced maps
\[ h^+(BG_k)\stackrel{i^+}{\to}
   h^+(BB_k)\stackrel{j(\bG)^+}{\to}
   h^+(\bG/B_k)\]
where $h^+(-)$ is the ideal of the positive degree parts.
Let us define 
\[D_{h}(\bG)=Ker(j^+)/(Ideal(Im(i^+))
%, \quad  \tilde  D_h(\bG)=D_{h}(\bG)/(S(t)^+\cdot D_h(\bG)).
\]

Let $\bG$ be versal and $k'$ is some extension of $k$.
Then
\[ D_h(\bG)\subset D_h(\bG|_{k'})\subset D_h(G|_{\bar k})\cong D_h(G_k).\]
 For ease of arguments we mainly  consider the case
$h^*(X)=CH^*(G)/p$, and write $D_h(\bG)$ simply by $D(\bG)$.

Recall
$gr S(t)/p\cong A(b)\otimes S(t)/(b)$
Let  $f_1,...,f_s\in A(b)$ and write by
\[A(b)(f_1,...,f_s)\quad (resp.\  S(t)(f_1,...,f_s))
\]
the ideal in $A(b)$ (resp.  $S(t)/p$) generated by $f_1,...,f_s$.  Then it is almost immediate
\begin{lemma}
We can write  additively
\[ S(t)(f_1,...,f_s)\cong A(b)(f_1,...,f_s)\otimes
S(t)/(b).\]
\end{lemma}
\begin{proof}
Each element $x\in S(t)(f_1,...,f_s)$ can be written as
\[ x=\sum_J( \sum_i b(J)_i f_i) t(J),\quad 
for\ b(J)_i\in A(b),\ \ 0\not =t(J)\in S(t)/(b).\]
\end{proof}
 \begin{lemma} Let $\bG$ be versal. Then 
there are maps
\[ D(G_k)/D(\bG)\subset CH^*(\bF)/p
\stackrel{pr}{\to}CH^*(R(\bG))/p,\]
such that  $pr(D(G_k)/D(\bG))=CH^+(R(\bG))/p$.
\end{lemma}
\begin{proof}
We consider the map
$ S(t)/p\cong CH^*(BB_k)/p\stackrel{j^*(\bG)}{\to}
CH^*(\bG/B_k)/p$.
   By the definition, we have
\[ D(G_k)/(D(\bG))\cong (Ker(j^*(G_k)/Im(i^+))/(Ker(j^*(\bG))/Im(i^+))\]
\[\cong Ker(j^*(G_k))/Ker(j^*(\bG)) 
\subset  S(t)/(Ker j^*(\bG))\cong  CH^*(\bF)/p
.\]

Recall that 
$CH^*(G_k/B_k)/p\cong P(y)\otimes S(t)/(b)$.
So $Ker(j(G_k))=(b)$.  From lemma 3.2, 
\[(b)=S(t)(b_1,...,b_{\ell}))\cong (A(b)(b_1,...,b_{\ell})
\otimes S(t)/(b)\cong A(b)^+\otimes S(t)/(b).\]
Since $pr S(t)/(b)=\bZ/p\{1\}$, we have the lemma.
\end{proof}

\begin{cor}  Let $\bG$ be versal.
Suppose there are $f_1(b),...,f_s(b)$ in $A(b)$ such that
\[CH^*(\bF)/p\cong S(t)/(p,f_1(b),...,f_s(b)).\]
Then $D(G_k)/D(\bG)\cong CH^+(R(\bG))/p\otimes
S(t)/(b).$
\end{cor}
\begin{proof}
The ideal $Ker(j^+(\bG))$ is written
\[ Ker j^*(\bG)\cong S(t)(f_1(b),...,f_s(b))
 \cong ( A(b)(f_1(b),...,f_s(b))\otimes S(t)/(b).\]
Hence we have
$ D(G_k)/D(\bG)\cong A(b)^+/(f_1(b),...,f_s(b))\otimes
   S(t)/(b).$
\end{proof}

\begin{cor}  Let $\bG$ be versal, and assume the supposition in  Lemma 2.2.  Moreover assume 
$Im(i^*)\subset A(b).$  Then there is 
$\tilde D(\bG)\subset D(\bG)$ such that
\[ D(\bG)\cong \tilde D(\bG)\otimes S(t)/(b).\]
\end{cor}
%Define $\tilde D(\bG)$ (even when the supposition of
%Lemma 2.2 is not satisfied) by
%\[ \tilde D(\bG)=D(\bG)/(S(t)^+/(b)\cdot D(\bG)).\]
%We will compute mainly this $\tilde D(\bG)$ in the following sections. 

From above corollaries,  we have a very weak version of
the decomposition theorem by Petrov-Semenov-Zainoulline [Pe-Se-Za], without using deep
theories of motives. 
\begin{cor}  Let $\bG$ be versal, and assume the supposition in  Lemma 2.2.  Then we have
an additive  decomposition of   the $mod(p)$ Chow ring
\[ CH^*(\bG/B_k)/p\cong S(t)/(b)\oplus
  D(G_k)/D(\bG) \]
\[\cong (\bZ/p\{1\}\oplus CH^+(R(\bG))/p)
\otimes S(t)/(b)\cong CH^*(R(\bG))\otimes S(t)/(b). \]
\end{cor}

%\begin{proof}  We recall the additive isomorphism
%\[ S(t)/p\cong (\bZ/p\{1\}\oplus A(b)^+)\otimes
%S(t)/(p,b_1,...,b_{\ell}).\]
%By considering the quotient module by
%$Ker(\bG_k)$, we have the first isomorphism.
%\end{proof}

%\begin{lemma}  If   $\bG$ satisfies the assumption of Lemma , then $G(G_k)/G(\bG)\cong CH^*(R(\bG))$.
%\end{lemma}

%\begin{cor}
%When $G$ is of type $(I)$, we have
%$G(G_k)/G(\bG)\cong CH^*(R(\bG))$.
%\end{cor}

{\bf Example.}   Let $G$ be of type $(I)$.
Then
\[Ker j^+(G_k)\cong Ideal(b_1,...,b_{\ell})\subset S(t)/p=CH^*(BB_k)/p,\]
\[ Ker j^+(\bG)\cong Ideal(b_ib_j,b_k| 1\le i,j\le 2p-2<k\le \ell)\subset S(t)/p.\]
Hence $D(G_k)/D(\bG)\cong
\bZ/p\{b_1,...,b_{2p-2}\}\otimes S(t)/(b)$.

\section{$CH^*(BG)/Tor$}

Since we have the Becker-Gottlieb transfer also in
$CH^*(X)$ by Totaro, we get the injection
\[ (4.1)\quad  CH^*(BG_k)/Tor\subset CH^*(BT)^{W}\]
for the Weyl group $W=N_G(T)/T$.  From \cite{YaW},
the above injection is always not surjective 
if $H^*(G)$ has $p$-torsion.
In general, to compute $CH^*(BG_k)$ is a difficult problem, but $CH^*(BG_k)/Tor$ seems  more computable.

Recall that  $gr_{geo}^*(X)$ (resp. $gr_{top}^*(X)$)
is the graded associated ring
defined by the geometric (resp. topological)
filtration of the algebraic $K$-theory $K^0_{alg}(X)$ (rep. the topological $K$-theory $K^0_{top}(X)$). 
Namely, it  is isomorphic to the infinite term $E_{\infty}^{2*,*,0}$ (resp. $E_{\infty}^{2*,0}$)
of the motivic (resp. usual) Atiyah-Hirzebruch spectral sequence.  Recall that $BP^*(-)$ is the $BP$-theory with the coefficient ring $BP^*=\bZ_{(p)}[v_1,v_2,...]$.

\begin{lemma}  There is an isomorphism
\[CH^*(BG_k)/Tor\cong gr_{geo}^*(BG_k)/Tor.\]
Moreover if $CH^*(BG_k)\to  gr_{top}^*(BG)/Tor$ (resp.
$(BP^*(BG)\otimes _{BP^*}\bZ_{(p)})/Tor$)   is surjective, then 
\[CH^*(BG_k)/Tor\cong gr_{top}^*(BG)/Tor\quad 
(resp.\  (BP^*(BG)\otimes_{BP^*}\bZ_{(p)})/Tor).\]
\end{lemma}
\begin{proof}
We consider the commutative diagram
\[ \begin{CD}
        CH^*(BG_k)/Tor @>(1)>> CH^*(BB_k) \\
           @V{(2)}VV      @V{\cong}VV  \\
         gr_{geo}^*(BG_k)/Tor  @>{(3)}>>  gr_{geo}^*(BB_k)
\end{CD} \]
There is the Becker-Gottlieb transfer, the map $(1)$
is injective.  Moreover the map $(2)$ is surjective,
and  we have the first isomorphism.
The second isomorphism follows from exchanging
$gr_{geo}^*(-)$ by $gr_{top}^*(-)$ (or by $BP^*(-)\otimes _{BP^*}Z_{(p)}$, and
$CH^*(BB_k)\cong gr_{top}^*(BT)$.
\end{proof}

On the other hand Totaro defines 
the modified cycle map $\bar cl$ such that
the composition $\rho \cdot\bar cl$
\[ (4.2)\quad  CH^*(X)\stackrel{\bar cl}{\to}
      BP^*(X)\otimes_{BP^*}\bZ_{(p)}\stackrel{\rho}{\to} H^*(X;\bZ_{(p)})\]
is the usual clycle map $cl$.  Moreover Totaro conjectures
that $\bar cl$ is isomorphic when $X=BG$ and $k=\bar k$.
More weakly, if the modified cycle map $\bar cl\ mod(Tor)$ is surjective, then 
we have  
$CH^*(BG_k)/Tor \cong (BP^*(BG)\otimes _{BP^*}\bZ_{(p)})/Tor$.  Note that we
have $CH^*(B\bar G_k)/Tor \cong 
CH^*(BG_k)/Tor$ in this case.

By arguments similar to the proof of Lemma 4.1,
(using $CH^*(B\bar B_k)\cong CH^*(BB_k)$)
we have
\begin{lemma}
If $res: CH^*(BG_k)/Tor \to CH^*(B\bar G_k)/Tor$ is surjective, then it is isomorphic.
\end{lemma}
\begin{cor}  Let $G$ be simply connected.
If $CH^*(B\bar G_k)/Tor$ is generated by
Chern classes, then
$res:CH^*(BG_k)/Tor\cong CH^*(B\bar G_k)/Tor.$
\end{cor}
\begin{proof}
When $G$ is simply conned, by Chevalley, we know
$res:K^0(BG_k)\cong K^0(B\bar G_k)$.
Hence  a map $B\bar G_k\to BU(N)$ can be lift
to a map $BG_k\to BU(N)$. This implies that
any Chern class in $CH^*(B\bar G_k)$ can be lift
to an element in $CH^*(BG_k).$
\end{proof}

\quad 

\section{$PGL(3)$ for $p=3$}

Now we consider in the case $(G,p)=(PU(p),p)$,
which has $p$-torsion in cohomology,
but it is not simply connected.
 Its mod $p$ cohomology is
\[ H^*(G;\bZ/p)\cong \bZ/p[y]/(y^p)\otimes
\Lambda(x_1,...,x_{p-1})\quad |y|=2,\ |x_i|=2i-1.\]
So $P(y)/p\cong \bZ/p[y]/(y^p)$ with $|y|=2$.
%This fact is given by the fibering $U(p)\to PU(p)\to BS^1$
%and the induced spectral sequence
%\[ E_2^{*,*'}\cong H^*(BS^1;H^{*'}(U(p);\bZ/p))
%\Longrightarrow H^*(PU(p);\bZ/p).\]
%Here we use that $H^*(BS^1;\bZ/p)\cong \bZ/p[y]$
%and $d_{2p}x_{p}=y^p$.

Since $G$ is not simply connected, $G$ is not of type $(I)$
while $P(y)$ is generated by only one $y$. (However
$CH^*(X)/p$ is quite  resemble to that of 
type $(I)$. Compare Theorem 2.4 and Theorem 5.2
below.)  

%We consider the map $U(p-1)\to U(p)\to PU(p)$ where
%the maximal tori of $U(p-1)$ and $PU(p)$ 
%are isomorphic, i.e., $T_{U(p-1)}\cong T_{PU(p)}$.
By using the map $U(p-1)\to PU(p)$, we know
$d_{2i}(x_i)=c_i$ for the elementary symmetric function
in $H^*(BT_{U(p)})$.
Then we have  
\[ grH^*(G/T;\bZ/p)\cong \bZ/p[y]/(y^p)\otimes
S(t)/(c_1,...,c_{p-1}).\]
\begin{lemma} We have $py^i=c_i\in H^*(G/T)_{(p)}$.
\end{lemma}

\begin{thm} Let $G=PU(p)$ and $X=\bG_k/B_k$.  Then there are isomorphisms
\[ CH^*(R(\bG_k))/p \cong CH^*(R_1)/p\cong \bZ/p\{1,c_1,...,c_{p-1}\},  \]
\[ CH^*(X)/p\cong S(t)/(p,c_ic_j|1\le i,j\le p-1).\]
\end{thm}

By Vistoli [Vi], it is known that $CH^*(BG)/Tor\cong
CH^*(BB_k)^W$.  However its ring structure is not 
mentioned  except for $p=3,5$.  (As additive groups it isomorphic to $\bZ_{(p)}[c_2,...,c_{p}]$, but they are not
isomorphic as  rings.)

We compute $D(\bG)$ only for $PU(3)$
\[CH^*(\bF)/3\cong S(t)/(3,c_1^2,c_1c_2,c_2^2).\]
By Vistoli and Vezzosi ( Theorem 14.2  in [Vi]), we have
\[ CH^*(BG_k)/Tor \cong \bZ_{(3)}[c_2',c_3',c_6']/(
27c_6'-4(c_2')^3-(c_3')^2).\]
Each  element $c_i'$ is written using $c_i$ in $(S(t)=CH^*(BB_k)$ (see page 48 in [Vi]) as 
\[ c_2'=3c_2-c_1^2,\quad c_3'=27c_3-9c_1c_2+2c_1^3,\quad c_6'=4c_2^3+27c_3^2\ mod(c_1).\]
Hence the map $i^*$ $mod(3)$ is given as
\[ c_2'\mapsto c_1^2,\quad
c_3'\mapsto c_1^3,\quad c_6'\mapsto c_2^3\ mod(c_1).\]
\begin{prop}
Let $(G,p)=PU(3),3)$ and $\bG$ be versal.  Then
\[D(\bG)\cong \bZ/3\{c_1c_2,c_2^2,c_1c_2^2\}\otimes
   S(t)/(c_1,c_2) \]
\end{prop}
\begin{proof}
The result follows form the quotient 
\[ (c_1^2,c_1c_2,c_2^2)/(c_1^2,c_1^3,c_2^3)\]
of ideals in $CH^*(B_k)/3\cong S(t)/3.$
\end{proof}

\quad

\section{$SO(2\ell+1)$}

At first we consider the
orthogonal groups $G=SO(m)$ and $p=2$.
The $mod(2)$-cohomology is written as ( see for example \cite{Tod-Wa}, \cite{Ni})
\[ grH^*(SO(m);\bZ/2)\cong \Lambda(x_1,x_2,...,x_{m-1}) \]
where $|x_i|=i$, and the multiplications are given by $x_s^2=x_{2s}$.

For ease of argument,  we only consider the case
$m=2\ell+1$ so that
\[ H^*(G;\bZ/2)\cong P(y)\otimes \Lambda(x_1,x_3,...,x_{2\ell-1}) \]
\[ grP(y)/2\cong \Lambda(y_2,...,y_{2\ell}), \quad 
letting\ y_{2i}=x_{2i}\ \ (hence \ y_{4i}=y_{2i}^2).\]

The Steenrod operation is given as 
 In particular,  $Sq^k(x_i)= {i\choose k}(x_{i+k}).$
The $Q_i$-operations are given by Nishimoto \cite{Ni}
\[Q_nx_{2i-1}=y_{2i+2^{n+1}-2},\qquad Q_ny_{2i}=0.\]

$Q_0(x_{2i-1})=y_{2i}$ in $H^*(G;\bZ/2)$.
It is well known that   the transgression 
$b_i=d_{2i}(x_{2i-1})=c_i$ is the  $i$-th elementary symmetric function
on $S(t)$.   Hence we have
\begin{lemma}   We have an isomorphism
\[grH^*(G/T)\cong P(y)\otimes S(t)/(c_1,...,c_{\ell}).\]
\end{lemma}

Moreover,  the cohomology   $H^*(G/T)$ is computed 
completely by
Toda-Watanabe \cite{Tod-Wa} (e.g. 
$2y_{2i}=c_i\ mod(4)$).
%\begin{thm} (\cite{Tod-Wa}) 
%There are $y_{2i}'\in H^*(G/T)$ for $1\le i\le \ell$
%such that $\pi^*(y_{2i}')=y_{2i}$ for $\pi: G\to G/T$
%(let us write $y_{2i}'$ also by $y_{2i}$ hereafter), and that 
%we  have an isomorphism
%\[ H^*(G/T)\cong \bZ[t_i,y_{2i}]/(c_i-2y_{2i},J_{2i})\]
%where $J_{2i}= 1/4(\sum_{j=0}^{2i}(-1)^jc_jc_{2i-j})$
%$=y_{4i}-\sum_{0<j<2i}(-1)^jy_{2j}y_{4i-2j}$ \\
%letting $y_{2j}=0$ for $j>\ell$.
%\end{thm}
%
%From from Corollary 2.1, we have
%\begin{cor} In $BP^*(G/T)/\II$, we have  
%\[c_i= 2y_{2i}+\sum_{n\ge 1} v_ny(2i+2^{n+1}-2) \]
%for some $y(j)$ with $\pi^*(y(j))=y_{j}$.
%\end{cor}
In $BP^*(G/T)/\II$, we have a relation from Lemma 2.1
and the result by Nishimoto
\[ c_i=2y_{2i}+v_1y_{2i+2}+...+v_jy_{2i+2(2^j-1)}+...\]

Let $T$ be a maximal Torus of $SO(m)$ and $W=W_{SO(m)}(T)$
its Weyl group.
Then $W\cong S_{\ell}^{\pm}$ is generated by permutations and change of signs so that $|S_k^{\pm}|=2^kk!$.
Hence 
we have
\[H^*(BT)^{W}\cong \bZ_{(2)}[p_1,...,p_{\ell}]\subset H^*(BT)\cong \bZ_{(2)}[t_1,...,t_{\ell}],\ |t_i|=2 \]
where the Pontriyagin class $p_i$ is defined by
$\Pi_i(1+ t_i^2)=\sum_ip_i$.

Here we recall
\[ H^*(BG;\bZ/2)\cong \bZ/2[w_2,w_3,...,w_{2\ell+1}], \quad  Q_0(w_{2i})=w_{2i+1} \ mod(w_sw_t).\]
It is known $H^*(BG)$ has no higher $2$-torsion
and
\[ H(H^*(BG;\bZ/2); Q_0)\cong (H^*(BG)/Tor)\otimes \bZ/2\]
where $H(A;Q_0)$ is the homology of $A$ with the differential $Q_0$.
Hence we have
\[ H^*(BG)/Tor \cong D\quad where\ \ 
D=\bZ_{(2)}[c_2, c_4,...,c_{2\ell}].\]
The isomorphism $j^*: H^*(BG)/Tor \to H^*(BT)^W$ 
is given by 
$ c_{2i}\mapsto p_i$.

Now we consider the $mod(2)$ Chow ring when $\bG$ is the split group $G_k$.
\begin{lemma} We have additive isomorphism
\[D(G_k)\cong \Lambda(c_1,..,c_{\ell})^+\otimes
S(t,c)\quad with\ S(t,c)\cong S(t)/(c_1,...,c_{\ell}).
\]
\end{lemma}
\begin{proof}
Recall that
\[CH^*(G_k/B_k)/2\cong H^*(G/T)/2\cong
P(y)/2\otimes S(t)/(c_1,...,c_{\ell}).\]
Hence we see 
\[Ker(j)\cong (c_1,...,c_{\ell})\subset
CH^*(BB_k)/2\cong H^*(BT)/2.\]
Here    $ j : p_i\mapsto  c_i^2$ $mod(2)$
by definition of the Pontryagin class $p_i$. 

On the other hand,
we know by Totaro [To1]
\[ CH^*(B\bar G_k)\cong 
BP^*(BG)\otimes _{BP^*}\bZ_{(2)}\cong\bZ[c_2,...,c_{2\ell+1}]/(2c_{odd}).\]
(In fact, $CH^*(B\bar G_k)/Tor\cong CH^*(BG_k)/Tor.$)
Hence $ CH^*(BG_k)/Tor\cong D\cong H^*(BT)^W$ by  $i:c_{2i}\mapsto p_i$. 
Thus the ideal generated by the image is 
 $(Im(i))\cong (c_2,c_4,...,c_{2\ell})\subset S(t)
$.   Since $j:p_i\mapsto c_i^2$, we have 
\[Ker(j)/(Im(i))\cong  (c_1,...,c_{\ell})/(c_1^2,...,c_{\ell}^2)\subset S(t)/(c_1^2,...,c_{\ell}^2).\]
It is additively isomorphic to
$\Lambda(c_1,...,c_{\ell})^+\otimes
 S(t)/(c_1,...,c_{\ell}),$
namely, each element $x\in D(G_k)$ is written
as $x=\sum _Ic(I)t(I)$ with $c(I)\in \Lambda(c_1,...,c_{\ell})^+$ and $t(I)\not =0\in
S(t)/(2,c_1,...,c_{\ell})$. 
\end{proof}

Recall that there is a surjection $D(G_k)\to CH^+(R(\bG))/p$ from Lemma 2.1.  We can see $c_1...c_{\ell}\not =0$
in $CH^*(R(\bG))/2$ (for example, using
the torsion index $t(G)=2^{\ell}$  [To2]).
\begin{thm} (Petrov \cite{Pe}, \cite{YaC}) 
Let $(G,p)=(SO(2\ell+1),2)$ and $\bF=\bG/B_k$
be versal.  
Then $CH^*(\bF)$ is torsion free, and 
\[ CH^*(\bF)/2\cong S(t)/(2,c_1^2,...,c_{\ell}^2),\quad
CH^*(R(\bG))/2\cong \Lambda(c_1,...,c_{\ell}).\]
\end{thm}

\begin{cor}   Let $(G,p)=(SO(2\ell+1),2)$ and 
$\bG$ be versal. 
 Then $D(\bG)\cong 0.$\end{cor}
\begin{proof}
We have
\  $ Ker(j^+)\cong (c_1^2,...,c_{\ell}^2)\cong Ideal(Im(i^+)).$
\end{proof}

\section{$BSpin(n)$ for $p=2$}

In this section, we study Chow rings for the cases $G=Spin(n)$,
$p=2$.
Recall that the $mod(2)$ cohomology is given by Quillen \cite{Qu}
\[ H^*(BSpin(n);\bZ/2)\cong \bZ/2[w_2,...,w_n]/J\otimes
\bZ/2[e]  \]
where $e=w_{2^h}(\Delta)$ and $J=(w_2,Q_0w_2,...,Q_{h-2}w_2)$.  Here $w_i$ is the Stiefel-Whitney class for the natural covering
$Spin(n)\to SO(n)$. The number $2^h$ is the Radon-Hurwitz number, dimension of the spin representation $\Delta$
(which is the representation $\Delta|_C\not =0$ for the center
$C\cong \bZ/2\subset Spin(n)$).  The element $e$ is the Stiefel-Whitney class $w_{2^h}$ of the spin representation $\Delta$.

Hereafter this section we always assume
$G=Spin(n)$ and $p=2$.
For the projection $\pi:Spin(n)\to SO(n)$, the maximal torus $T$
of $Spin(n)$ is given $\pi^{-1}(T')$ for the maximal torus $T'$ of $SO(n)$, and $W=W_{Spin(n)}(T)\cong
W_{SO(n)}(T')$.   Benson-Wood [Be-Wo] determined
$H^*(BT)^W$ and proved
\begin{thm} (Benson-Wood Corollary 8.4 in [Be-Wo])
Let $G=Spin(n)$ and $p=2$.
Then $\rho^*_{H}:H^*(BG)\to H^*(BT)^W$ is surjective
if and only if $n\le 10$ or  $n\not =3,4,5\ mod(8)$ (i.e., 
it is not the quaternion case).
\end{thm}

Moreover,  in this section, we assume 
$Spin(n)$
is in the real case \cite{Qu}, that is $n=8\ell-1,  8\ell+1$
(hence $\rho^*_H$ is surjective and $h=4\ell-1,4\ell$ respectively). 

Benson and Wood define invariants  $q_i$, $\eta_{\ell-1}$ such that 
\[(1)\quad q_1=1/2p_1, \quad  q_i^2=2q_{i+1} \quad with 
\ |q_i|=2^{i+1},\]
\[(2)\quad \eta_{\ell-1}^2=\rho^*(c_{2^h}(\Delta_{\bC}))=\rho^*(e^2),\quad |\eta_{\ell-1}|=2^h.\]
In fact  in $H^*(BT)^W$, it is defined as
$ \eta_{\ell-1}=\Pi_{I\subset \{2,...,\ell\}}
(q_1-(\Sigma_{i\in I}x_i)).$

Then Benson-Wood prove 
\begin{thm} ( Theorem 7.1 in [Be-Wo])  If 
 $n=2\ell+1\ge 7$, then
\[ H^*(BT)^{W}\cong \bZ_{(2)}[p_2,...,p_{\ell},\eta_{\ell-1}]
\otimes \Lambda_{\bZ}(q_1,...,q_{\ell-2})\]
where $\Lambda_{\bZ}(a_1,...,a_k)$ is the free module
generated by $a_1^{\epsilon_1}...a_k^{\epsilon_k}$ for $\epsilon_i=0,1.$
\end{thm}

On the other hand, by Kono \cite{Ko},  $H^*(BG;\bZ)$ has no
higher $2$-torsion,  \[ H(H^*(BG;\bZ/2);Q_0)\cong (H^*(BG)/Tor)\otimes \bZ/2.\]
 Benson and Wood also define  $s_i\in H^*(BSO(n);\bZ/2)$  such that
\[ Q_0(s_i)=Q_i(w_2)\quad mod(s_1,...,s_{i-1})\]
and hence $s_i\in H(H^*(BG;\bZ/2);Q_0)$.
So we can identify $s_i \in H^*(BG)/Tor$.

\begin{cor} ([Be-Wo])
The cohomology $H^*(BG)/Tor$ is isomorphic
\[ D_{\ell}\otimes \Lambda_{\bZ}(s_3,...,s_{\ell},e)
\quad with \ D_{\ell}=\bZ_{(2)}[c_4,c_6,...,c_{\ell},c_{2^h}]
\] 
where $c_i=w_i^2$ are  lifts in $H^*(BG;\bZ)/Tor$ of
the same named elements in $H^*(BG;\bZ/2)$.
\end{cor}
 The map $i^*$ is given  with modulo (decomposed elements.e.)
\[  c_{2i}\mapsto p_i,\quad e\mapsto \eta_{\ell-1},\quad 
s_i\mapsto q_{i-2}.\]

For actions of   $Q_i$ on $H^*(BG;\bZ/2)$,
we use the following lemma
\begin{lemma} (Inoue)  
Let us write  $(W)=\bZ/2[w_2,...,w_n]^+$.
In $H^*(BSO(N);\bZ/2)$. we have
\[ (1)\quad Q_i(w_j)=\begin{cases}
                w_{j+2^i-1}\quad  mod(W^2)\quad if\ j=even \\
0\ mod(W^2) \quad j=odd.\end{cases}\]
\[ (2)\quad when \ N<2^{i+1}-1+j,\ \ Q_i(w_j)=w_jw_{2^{i+1}-1}\  mod(W^3).\]
\end{lemma}

\begin{lemma}
Let $2^i<2\ell+1$.  Then we can take $s_{i-1}=w_{2^{i}}\
mod(W^2)$.  The element 
$s_{i-1}$ is not in the image of the cycle map  
from the  Chow ring.
\end{lemma}
\begin{proof} 
By Inoue's lemma,
\[ Q_0(s_i)=Q_i(w_2)=w_{2^{i+1}+1}\ mod(W^2).\]
Hence $s_i=w_{2^{i+1}}$ $mod(W^2)$.

Since $Q_i(x)=0$ for each class $x$ in the 
$mod(2)$ Chow ring, the second statements follows from
\[Q_1(w_{2^{i+1}})=w_{2^{i+1}+3}\not \in J\ mod(W^2)
\quad
when \ 2^i<2\ell-1.\]
  For $2^i=2\ell$, we have $
Q_{i}(w_{2^{i+1}})=w_{2^{i+1}-1}w_{2^{i+1}}\not \in J\ mod(W^3).$
\end{proof}

In our (real)  case,
it is known \cite{Qu} that each maximal elementary abelian $2$-group $A$ has 
$rank_2A=h+1$ and
$ e|A=\Pi _{x\in H^1(B\bar A;\bZ/2)}(z+x)
$.  Here  we identify 
$A\cong C\oplus \bar A$ and
\[ H^*(BC;\bZ/2)\cong \bZ/2[z], \quad H^*(B\bar A;\bZ/2)\cong \bZ/2[x_1,...,x_h].\] 
The Dickson algebra is written as a polynomial algebra
\[ \bZ/2[x_1,....,x_h]^{GL_h(\bZ/2)}
\cong \bZ/2[d_0,....,d_{h-1}].\]
where $d_i$ is defined as
$ e|A=z^{2^h}+d_{h-1}z^{2^{h-1}}+...+d_0z.$
We can  also identify $d_i=w_{2^h-2^i}(\Delta)
\in H^*(BG;\bZ/2)$ \cite{Qu}.
%\begin{lemma}  (Lemma 2.1 in \cite{Sc-Ya})
%Milnor operations act on $d_i$ by
%\[Q_{h-1}d_i=d_0d_i,\quad Q_{j-1}d_j=d_0,\ for\ 1 \le j,\]
%\[Q_id_j=0\quad  for\ i<h-1\ and \ i\not =j-1.\]
%\end{lemma}

\begin{lemma} (Corollary 2.1 in \cite{Sc-Ya})
We have
\[Q_{h-1}e=d_0e\quad  and\quad Q_ke=0\ \  for\ 0\le k\le h-2.\]
\end{lemma}
Thus we know that $e=\eta_{\ell-1}$ is not in the image from $CH^*(BG)$.
Let us consider 
$i^*/2:CH^*(BG)\to CH^*(BT)/2$
(but not to $CH^*(BT)^W/2$).
\begin{conj}  Let $G=Spin(2\ell+1)$ be of real type.
\[  Im(i^*/2(CH^*(BG_k))\cong 
 D_{\ell}/2=\bZ/2[c_4,c_6,..., c_{2\ell},
 c_{2^h}] \subset H^*(BT)/2.\]
\end{conj} 
We will see that the above conjecture is true
when
$G=Spin(7),Spin(9)$,  and some weaker version for $Spin(\infty)$.

We consider the motivic cohomology so that
\[ CH^*(X)/2\cong H^{2*,*}(X;\bZ/2).\]
The degree is given $deg(w_i)=(i,i)$ and $deg(c_i)=(2i,i)$.
The cohomology operation $Q_i$ exists in the motivic cohomology with $deg(Q_i)=(2^{i+1}-1, 2^i-1)$. Hence 
\[ Q_iQ_0(w_2)\in H^{2*,*}(BG_k;\bZ/2)\cong CH^*(BG_k)/2.\]
Using these facts, we can see
\begin{thm}  (\cite{YaCo}) 
The ring $CH^*(BSpin(n)_k)/2$
has a subring
\[ RQ(n)=\bZ/2[c_2,...,c_{n}]/(Q_1Q_0w_2,...,Q_{n-1}Q_0w_n)
\otimes \bZ/2[c_{2^h}(\Delta_{\bC})]\]
where $c_i$ is the Chern class for $Spin(n)\to SO(n)\to 
U(n)$ and $c_{2^h}(\Delta_{\bC})$ is that of complex representation for $\Delta$. 
\end{thm}
\begin{proof}  This theorem is proved in \cite{YaCo}
for $k=\bar k$.  It is well known $K^*(BG_k)\cong K^*(BG_{\bar k})$.  Hence we see all
Chern classes in $CH^*(BG_{\bar k})$ can be extended
for $CH^*(BG_k)$.  (see Corollary 4.3.)
\end{proof}

\begin{lemma} Let $m=2\ell+1$ and $G$ be real type.  Then we have the 
isomorphism
\[ i^*(RQ(m))\cong D_{\ell}=\bZ_{(2)}[c_4,c_6,...,c_{2\ell},c_{2^{h}}(\Delta_{\bC})].\]\end{lemma}
\begin{proof}
The relation $Q_0Q_jw_2$ exists in $CH^*(BG(m))/2$.
 The element 
\[ c_{2i+1}=w_{2i+1}^2=Q_0(w_{2i})w_{2i+1}=Q_0(w_{2i}w_{2i+1})\]
also exists in $CH^*(BG(m))$ and $2$-torsion.
\end{proof}

\section{$\bG/B_k$ for $G=Spin(n)$}

In this section, let    $G=SO(2\ell+1)$
and $G'=Spin(2\ell+1)$. (These notations differ from the preceding section.)
It is well known that
$ G/T\cong G'/T'$ for the maximal torus $T'$ of the spin group.

  By definition, we have the 
$2$ covering $\pi:G'\to G$.
It is well known that 
$\pi^*:   H^*(G/T)\cong H^*(G'/T')$.
Let $2^t\le \ell < 2^{t+1}$, i.e. $t=[log_2\ell]$.
The mod $2$ cohomology is
\[ H^*(G';\bZ/2)\cong H^*(G;\bZ/2)/(x_1,y_2)\otimes
\Lambda(z) \]
\[ \cong P(y)'\otimes \Lambda(x_3,x_5,...,x_{2\ell-1})\otimes \Lambda(z),\quad |z|=2^{t+2}-1\]
where
$P(y)\cong \bZ/2[y_2]/(y_2^{2^{t+1}})\otimes P(y)'$.  
%(Here $z$ is defined by $d_{2^{t+2}}(z)=y^{2^{t+1}}$ for $0\not =y\in H^2(B\bZ/2;\bZ/2)$ in the spectral sequence induced from
%the fibering $G'\to G\to B\bZ/2$.)
Hence
\[ grP(y)'\cong \otimes _{2i\not =2^j}\Lambda(y_{2i})\cong
\Lambda(y_6,y_{10},y_{12},...,y_{2\bar \ell})\]
where $\bar \ell=\ell-1$ if $\ell=2^j$ for some $j$, and
$\bar \ell=\ell$ otherwise.

The $Q_i$ operation for $z$ is given by Nishimoto 
\cite{Ni}
\[ Q_0(z)=\sum _{i+j=2^{t+1},i<j}y_{2i}y_{2j}, \quad 
 Q_n(z)=\sum _{i+j=2^{t+1}+2^{n+1}-2,i<j}y_{2i}y_{2j}\ \ for\ n\ge 1.\]

We know that 
%\[ grH^*(G/T)/2\cong P(y)'\otimes \bZ[y_2]/(y_2^{2^{t+1}})\otimes S(t)/(2,c_1,c_2,...,c_{\ell}) \]
\[ grH^*(G'/T')/2\cong P(y)'\otimes S(t')/(2,c_2',.....,c_{\ell}',c_1^{2^{t+1}}).\]
Here $c_i'=\pi^*(c_i)$ and $d_{2^{t+2}}(z)=c_1^{2^{t+1}}$ in the spectral sequence
converging $H^*(G'/T')$.
%These are isomorphic, in  particular, we have
%\begin{lemma}  The element $\pi^*(y_2)=c_1\in S(t')$
%and $\pi^*(t_j)=c_1+t_j$ for $1\le j\le \ell$.
%\end{lemma}
Take $k$ such that $\bG$ is a versal $G_k$-torsor so that
$\bG'_k$ is also a versal $G_k'$-torsor.  Let us write
$\bF=\bG/B_k$ and $\bF'=\bG'/B_k'$.  Then
$ CH^*(\bar R(\bG'))/2\cong P(y)'/2$.

The Chow ring $CH^*(R(\bG'))/2$ is not computed yet
(for general $\ell$),
while we have the following lemmas.
%\begin{lemma}
%Let $2^t\le \ell<2^{t+1}$and $\bG'$ be versal. 
%Then there  is a surjection
%\[ \Lambda(c_2',...c_{\bar \ell}')\otimes \bZ/2[c_1^{2^{t+1}}]\to CH^*(R(\bG'))/2. \]
%\end{lemma}

%\begin{lemma}  The restriction $res_K:
%K^*(\bF')\to K^*(\bar \bF') $ 
% is isomorphic.
%In fact, we have $K^*(R(\bG'))/2\cong 
%K^*/2\otimes  \Lambda(c_i'|i\not =2^j-1).$
%\end{lemma}

%In general,  $gr_{\gamma}(R(\bG))$ seems complicated.   
%Let us write $\bZ_{(p)}$-free module
%\[ \Lambda_{\bZ}(a_1,...,a_n)=\bZ_{(p)}\{a_{i_1}...a_{i_s}|
%1\le i_1<...<i_s\le n\}.\]

\begin{lemma}  Let $G=Spin(2\ell+1)$ and 
$2^t\le \ell<2^{t+1}$.  Then
there is a surjection
\[ \Lambda(c_2',...,c_{\bar \ell}')\otimes
\bZ/2[e_{2^{t+1}}]\to
CH^*(R(\bG))/2.\]
where $c_i'=\pi^*(c_i)$ and $e_j=c_1^j$ in $S(t)
\cong H^*(BT)$ for $\pi:Spin(m)\to SO(m)$. 
\end{lemma}

\begin{lemma}
We have 
\[ i^*(c_{2i})=(c_i')^2,\quad i^*(c_{2i+1})=0,\quad   i^*(c_{2^h}(\Delta_{\bC}))=e_{2^{t+1}}^{2^{h-t-1}}.
\]\end{lemma}
\begin{proof}
The first equation is well known (see Lemma 7.3
in [YaC]), in fact $c_i^2=0$ in $CH^*(\bG)$ is proved
using $CH^*(BG_k)$ for $G=SO(n)$.  The second
equation follows from $P^1c_{2i}=c_{2i+1}$ and $P^1((c_{i}')^2)=0$.
The last equation follows from the fact $\Delta$ is spin representation(which is nonzero in the restriction on $\bZ/2$ (recall $e_{2^{t+1}}=c_1^{2^{t+1}})$.
\end{proof}

Hereafter in this section,   let us 
\[ write\ \ c_i'\ \ by\ \  c_i ,\quad and \quad 
G=Spin(2\ell+1)\]
 as the notations
in the preceding section.
\begin{cor} 
% When $\Delta$ can be extended over $k$
There is a surjection
\[ \Lambda(c_2,...,c_{\bar \ell})\otimes
\bZ/2[e_{2^{t+1}}]/(e_{2^{t+1}}^{2^{h-t-1}})\to
CH^*(R(\bG))/2.\]
\end{cor}

\begin{lemma}
Let $G(n)=Spin(n)$ and $\bG(n)$ be
the versal $G(n)_k$-torsor.  Given $n\ge 1$, there is 
$N\ge 7$ such that
\[  CH^*(R(\bG(N))/2\cong \Lambda (c_2,...,c_n)\quad
for\ all \ *\le  n.\]
\end{lemma} 
\begin{proof}
Let $N=2\ell+1$,  and
$ 2^2\le 2^n<2^t\le \ell<2^{t+1}.$  \\
We will  see 
\[  CH^*(R(\bG(N))/2\cong \Lambda (c_2,...,c_{\ell})
\quad for \  *< 2^n. \]

Suppose that
\[ x= \sum  c_{i_1}... c_{i_s} =0\in CH^*(R(\bG))/2\quad
for \ \ 2\le  i_1<...<i_s<2^n.\]
Recall $k(n)^*=\bZ/p[v_n]$ and $k(n)^*(\bar R(\bG))
\cong k(n)^*\otimes P(y)$.
We note that in $k(n)^*(\bar R(\bG)$ 
\[c_{i_j}=v_ny_m\quad with\ \   m={2^n-2+2i_j}.\]
  Since $2^n<m<2^{n+1}$,
$m$  is not a form $2^r$, $r>3$. Hence $y_m$
 is a generator of $P(y)$.

Moreover  recall that  
\[ e_{2^{t+1}}=v_ny_{2^{t+1}-2+2^n-2}y_{2^{t}+2}+...\]
This element  is in the $ideal(v_n^2,E)$ with
$E=(y_{2j}|j> 2^{t}).$
Hence we see $c_{i,j}=v_ny_m$ is  also nonzero $mod(v_n^2,E)$ since $n<t$.

Thus we see that 
for $x'= y_{2^n-2+2{i_1}}....y_{2^n-2+2i_s}$,
which is an additive generator of $P(y)$
\[ x=\sum v_n^sx'\not =0 \in k(n)^*(\bar R(\bG))/2\cong k(n)^*\otimes P(y).\]
Moreover $v_n^{s-1}x'\not \in Im(res)$, because 
$Im(res)$ 
 is generated by $res(c_{j_1})...res(c_{j_r})$
and each $res(c_j)=0$ $mod(v_n)$.
This is a contradiction.
\end{proof}

\begin{cor}  Let $G(N)=Spin(N)$ and $\bG(N)$ be versal.  Then we  have
\[  lim_{\infty \gets N}CH^*(R(\bG(N))/2\cong \Lambda(c_2,c_3,...,c_{n},...),\]
\[  lim_{\infty \gets N}CH^*(\bF)/2\cong S(t)/(2,c_2^2,c_3^2,...,c_{n}^2,...),\]
\end{cor}
\begin{proof}  The second isomorphism follows from
the additive isomorphism
\[CH^*(\bF)/2\cong CH^*(R\bG(N))/2\otimes S(t)/(c_2,c_3,...).\]
\end{proof}
\begin{cor}
We have \
$  lim_{\infty \gets N}D(\bG(N))=0. $
\end{cor}
\begin{proof}
From Lemma 7.9, we have
\[  lim_{\infty \gets N} Ideal(i^*/2CH^*(BG(N)_k)  
 \supset  ( D_{\infty}/2)\]
\[= Ideal(c_4,c_6,...,c_{2i},...)\subset CH^*(BB_k)/2.\]
We get the result from $c_{2i}\to c_i^2$ by the map $i^*$.
\end{proof}

%{\bf Example.}  Let $G=Spin(9)$ and $\bG$ be versal.
%The Chow ring of the flag variety is given in $\S 6$ and
%\[Ker(j^*(\bG))=(2c_2,c_2^2,c_2c_3,c_3^2,e_8,c_4)\subset S(t),\]
%The Chow ring of $BG$ is still unknown.  But it is known
%in (6.14 in [Ko-Ya])
%\[ CH^*(BG_{\bar k})/(Tor)\cong (BP^*(BG_{\bar k})\otimes_{BP^*}\bZ_{(2)})/(Tor)\]
%\[ \cong \bZ_{(2)}[c_4,c_6,c_8,c_{16}]\{1,c_2'',c_4'',c_6'',c_8'',c_{10}'',c_{12}''\}\]
%Then we can see the map $i^*$ is given 
%$ c_4\mapsto c_2^2,$  $c_6\mapsto c_3^2$,
%$ c_8\mapsto e_8$,  $c_{16}\mapsto (c_4)^4$, and 
%\[ c_2''\mapsto 2c_2,\ \ c_4''\mapsto 2c_4,\ \
%c_6''\mapsto 2c_2c_4,\ \ 
% c_8''\mapsto 2c_4^2,\ \ c_{10}''\mapsto 2c_2c_4^2,\ \
%c_{12}''\mapsto 2c_4^2e_8.\]
%Thus we have
% $ D_{CH}(\bG)\cong ((\bZ\{1\}\oplus \bZ/2\{c_2c_3\})\otimes \bZ[c_4]/(2c_4^2,c_4^4))^+.$

\section{ $Spin(7)$ for $p=2$} 

Hereafter this section, we assume $G=Spin(7)$ and $p=2$.
It is well known
\[ H^*(BG;\bZ/2)\cong \bZ/2[w_4,w_6,w_7,w_8]\]
 where $w_i$  for $i\le 7$ (resp. $i=8$)
are the Stiefel-Whitney classes for the representation
induced from $Spin(7)\to SO(7)$
(resp. the spin representation $\Delta$).

Thus the integral cohomogy is written as
(using $Q_0w_6=w_7$)
\[ H^*(BG)\cong \bZ_{(2)}[w_4, c_6,w_8]\otimes (\bZ_{(2)}\{1\}
\oplus \bZ/2[w_7]\{w_7\})\]
\[ \cong D\otimes \Lambda_{\bZ}(w_4,w_8)
\otimes (\bZ_{(2)}\{1\}
\oplus \bZ/2[w_7]\{w_7\})\]
where $D=\bZ_{(2)}[c_4,c_6,c_8]$ with $c_i=w_i^2$.

Next we consider the Atiyah-Hirzebruch spectral
sequence
\[ E_2^{*,*'}\cong H^*(BG)\otimes BP^*
\Longrightarrow BP^*(BG).\]
We can compute the spectral sequence
\[ grBP^*(BG)\cong D\otimes (BP^*
\{1,2w_4,2w_8,2w_4w_8,v_1w_8\}\]
\[ \oplus BP^*/(2,v_1,v_2)[c_7]\{c_7\}/(v_3c_7c_8)).\]
Then $BP^*(BG)\otimes_{BP^*}\bZ_{(2)}$
is isomorphic to ([Ko-Ya])
\[ D\{1,2w_4,2w_8,2w_4w_8,v_1w_8\}/(2v_1w_8)
 \oplus D/2[c_7]\{c_7\}.\]

On the other hand, the Chow ring of $BG_{\bC}$ is given
by Guillot (\cite{Gu},\cite{YaF4})
\begin{thm}  Let $k=\bar k$.  Then we have isomorphisms
\[ CH^*(BG_k)\cong BP^*(BG_k)\otimes_{BP^*}\bZ_{(2)}\]
\[ \cong D
\otimes(\bZ_{(2)}\{1,c_2',c_4'.c_6'\}
\oplus \bZ/2\{\xi_3\}\oplus \bZ/2[c_7]\{c_7\})\]
where $cl(c_i)=w_i^2$, $cl(c_2')=2w_4$, $cl(c_4')=2w_8$, $cl(c_6')=2w_4w_8$,  and $cl(\xi_3)=0$, $|\xi_3|=6$.
However  $cl_{\Omega}(\xi_3)=v_1w_8$ in $BP^*(BT)^W$,
for the cycle map $cl_{\Omega}$ of the algebraic 
cobordism.  
\end{thm}

Now we consider $CH^*(\bG/B_k)$.
   Let $G=Spin(7)$ and $\bG$ be versal.
The group $G$ is of type $(I)$ and we can take
$b_1=c_2,b_2=c_3, b_3=e_4$ with $|e_4|=8$.

The Chow ring $CH^*(\bG/B_k)$ is given in
Theorem 2.3 
(in fact, $G$ is of type $(I)$) 
\[ CH^*(\bG/B_k)\cong S(t)/((2c_2,c_2^2,c_2c_3,c_3^2,e_4),\quad S(t)=\bZ_{(2)}[t_1,t_2,t_3].\]
Hence we have $Ker(j(\bG))\cong (2c_2,c_2^2,c_2c_3,c_3^2,e_4).$  Recall
\[CH^*(BG_{\bar k})/(Tor)\cong CH^*(BB_k)^W\cong D\{1,c_2'',c_4'',c_6''\}\]
where  $c_i''$ is a Chern class of the (complex) 
spin 
representation.  
Note $CH^*(BG_{\bar k})/Tor\cong
CH^*(BG)/Tor$ from Lemma 4.3.
Since $i(c_2'')=2w_4,...$, we see
\[ D/2\cong Im(i^*/2: CH^*(BG_k)\to CH^*(BT)/2).\]

We can see that the map $i^*$ is given 
$ c_4\mapsto c_2^2,$ $c_6\mapsto c_3^2$,
$ c_8''\mapsto  e_4^2,$ and 
\[ c_2''\mapsto 2c_2,\ \ c_4''\mapsto 2e_4,\ \
c_6''\mapsto 2c_2e_4.\]
In particular $i^*CH^*(BG_k)=i^*CH^*(BG_{\bar k})$.
Thus we see 
\begin{prop}  Let $G=Spin(7)$ and $\bG$ be versal.
Then  we have additively
\[ D(\bG)\cong \Lambda(c_2c_3,e_4)^+\otimes
S(t,c)\quad for\ S(t,c)\cong S(t)/(c_2,c_3,e_4).\]
\end{prop}

\section{$Spin(9)$ for $p=2$} 

In  this section, we assume $G=Spin(9)$ and $p=2$
and hence $h=4$.
It is well known (in fact $w_2,w_3,w_5\in J$)
\[ H^*(BG;\bZ/2)\cong \bZ/2[w_4,w_6,w_7,w_8,w_{16}]\]
where $w_i$  for $i\le 8$ (resp. $i=16$)
are the Stiefel-Whitney class for the representation
induced from $Spin(9)\to SO(9)$
(resp. the spin representation $\Delta$
and hence $w_{16}=w_{16}(\Delta)=e$).

Recall that  $H^*(BG)$ has just $2$-torsion by Kono.
Let us write
\[D=\bZ_{(2)}[c_4,c_6,c_8,c_{16}]\quad with\ c_i=w_i^2.\]
Then we can write
\[ H^*(BG)/Tor\cong D\otimes \Lambda(w_4,w_8,w_{16}), \quad 
 Tor\cong D\otimes \bZ/2[w_7]^+.\]

Next we consider the Atiyah-Hirzebruch spectral
sequence
\[ E_2^{*,*'}\cong H^*(BG)\otimes BP^*
\Longrightarrow BP^*(BG).\]
Using $Q_1(w_4)=w_7, Q_2(w_7)=c_7$,  $Q_2(w_8)=w_7w_8$
and $Q_3(w_7w_8)=c_7c_8,$
we can compute the spectral sequence
(page 796, (6.14) in [Ko-Ya]).
Let us write $D'=\bZ_{(2)}[c_4,c_6,c_8]$ and $D''=\bZ_{(2)}[c_4,c_6,c_{16}]$.  Then 
the infinite term
is given
\[E_{\infty}=grBP^*(BG) \]
\[ \cong  D'\otimes (BP^*\{1,2w_4,2w_8,2w_4w_8,v_1w_8\}
\oplus  BP^*/(2,v_1,v_2)[c_7]^+/(v_3c_7c_8))\]
\[\oplus D'' \otimes (BP^*\{2w_4w_{16}, v_1w_{16},v_2w_{16}\}  
  \oplus  BP^*/(2,v_1,v_2)[c_7]\{c_7c_{16}\})\]
\[  \oplus D\otimes (BP^* \{2w_8,2w_4w_8,v_1w_8
\}\{w_{16}\}
 \oplus  BP^*/(2,v_1,v_2,v_3,v_4)[c_7]\{c_7c_8c_{16}\}).\]
However  $BP^*(BG)\otimes_{BP^*}\bZ_{(2)}$
is not so complicated, and it is isomorphic to 
\[  BP^*(BG)\otimes_{BP^*}\bZ_{(2)}\ \cong \ D\{1\}\oplus D\otimes 2\Lambda_{\bZ}(w_4,w_8,w_{16})^+\]
\[ \oplus D/2\{ v_1w_8, v_1w_{16},v_1w_8w_{16},v_2{w_{16}}\}
\oplus  D/2[c_7]\{c_7\} .\]
%Here note $v_1w_8$ is torsion free  in $BP^*(BG)$,
%however  $2w_8\in BP^*(BG)$ and 
%\[ 2 v_1w_8=v_1(2w_8)=0\quad in\ BP^*(BG)\otimes _{BP^*}\bZ_{(2)}.\]
 
The elements in $BP^*(BG)$
corresponding to $v_1w_8,...,v_2w_{16}$ are all 
torsion free elements.  However they are $2$-torsion 
in $BP^*(BG)\otimes _{BP^*}\bZ_{(2)}$,
e.g., 
\[ 2v_2w_{16}\in v_2BP^*(BG),\quad since\ 2w_{16}
\in BP^*(BG).\]

We will prove the following lemma.
\begin{lemma} Each element in $2\Lambda_{\bZ}(w_4,w_8,w_{16})$ is represented by
a sum of products of  Chern classes.
\end{lemma}

Hence $\tilde cl/Tor$ is surjective.  So from Lemma 4.1,
we have

\begin{thm}  We have the isomorphism
\[ CH^*(BG)/(Tor)\cong (BP^*(BG)\otimes_{BP^*}\bZ_{(2)})/(Tor)\]
\[ \cong D\{1,c_2'',c_4'',c_6'',
c_8'',c_{10}'',c_{12}'',c_{14}''\}\]
where $c_i$ (resp. $c_j''$) is the Chern class of
the usual (resp. complex spin) representation.\end{thm}

Let us write by $Grf\subset CH^*(BG)$ be the
ideal of  Griffiths elements, that is
$Grf=Ker (cl: CH^*(BG)\to H^*(BG)).$

\begin{cor}  We have 
$ Tor/Grf\cong  D/2[c_7]\{c_7\}.$
\end{cor}

{\bf Remark.}
Note that $v_1w_8\in Grf$, but we can not see
$v_1w_{16},v_2w_{16}$ are in $CH^*(BG)$ or not, 
 i.e., we do not see $\bar cl$ is surjective or not.

To prove the above lemma, we recall the complex representation ring
\[R(Spin(2\ell+1))\cong 
\bZ[\lambda_1,...,\lambda_{\ell-1},
\Delta_{\bC}] \]
Here $\lambda_i$ is the $i$-th elementary symmetric function in variables $z_1^2+z_1^{-2}$,...,$z_{\ell}^2+z_{\ell}^{-2}$ in $R(T)\cong \bZ[z_1,
z_1^{-1},...,z_{\ell},z_{\ell}^{-1}]$ for the maximal torus $T$.
The representation $\Delta_{\bC}$ is defined
\[ \sum z_1^{\epsilon_1}...z_{\ell}^{\epsilon_{\ell}}\quad 
\epsilon_{i}=1\ or\ -1.\]

Consider the restriction $R(S^1)\cong \bZ[z_1,z_1^{-1}]$
(i.e., $z_i=1$ for $i\ge 2$).  
Since
\[\lambda_1=z_1^2+z_1^{-2}+...+z_{4}^2+z_{4}^{-2},\quad so\ \ \lambda_1|{S^1}=z_1^2+z_1^{-2}+6.\]
 Thus for $ H^*(BS^1)\cong \bZ[u],\ |u|=2,$  we have
\[Res_{BS^1}(c(\lambda_1))
=(1-2u)(1+2u)=1-4u^2.\] 
From this 
we see $c_i|_{S^1}=c_i(\lambda_1)|_{S^1}=0$ for $i>2$.
Note $Res_{S^1}(w_4)=0$ in $H^*(BS^1;\bZ/2)
$, but
$w_4$ is not represented by Chern  
(in fact, it does not exist in $BP^*(BG)$.).
Using this, we can see
\[
Res_{BS^1}(w_4)=2u^2\quad
and\ so \quad Res_{BS^1}(2w_4)=4u^2\]
which  is represented by Chern classes.

\begin{proof}[Proof of Lemma 10.1.]
We consider the Chern classes $c_i(\Delta_{\bC})|_{BS^1}$.
Consider the restriction
$ \Delta_{\bC}|_{S^1}=2^{3}( z_1+z_1^{-1})$.
Hence 
\[Res_{BS^1}(c(\Delta_{\bC}))=(1-u^2)^{8}=
1-\binom{8}{1}u^2+\binom{8}{2}u^4+...+u^{16}.\]

Recall $q_3|_{BS^1}=w_4|_{BS^1}=2u^2$.  Since $q_4^2=2q_{3}$,  we see $w_8|_{BS^1}=q_4|_{BS^1}=2u^4$.
We also know $e|_{BS^1}=u^4$ (in fact $e=w_{16}$ is defined 
using $\Delta$).
Therefore $2w_8|_{BS^1}=4u^4$ and  $2e|_{BS^1}=2u^{8}$ are represented by Chern classes.  Similarly we can see that each element in
$2\Lambda_{\bZ}(w_4,w_8,w_{16})$ is represented by Chern class.  For example
\[ 2w_4w_8w_{16}|_{S^1}=2(2u^2)(2u^4)u^8=2^3u^{14}=\binom{8}{7}u^{14}\]
which is represented by a Chern class.
\end{proof}

  Let $G=Spin(9)$ and $\bG$ be versal.
The Chow ring of the flag variety is given in $\S 6$ and
\[Ker(j^*(\bG))=(c_2^2,c_2c_3,c_3^2,e_8,c_4)\subset S(t)/2,\]
The Chow ring of $BG$ is still unknown.  But
we see from the preceding theorem
$CH^*(BG_k)/Tor \cong D\{1,c_2'',c_4'',c_6'',c_8'',c_{10}'',
c_{12}'',c_{14}''\}.$
Since $i^*(c_2'')=2w_4,i^*(c_4'')=2w_8,...$, we see
Conjecture 6.7 in the preceding section
\begin{thm}  Let $G=Spin(9)$.  Then for 
$D=\bZ_{(2)}[c_4,c_6,c_8,c_{16}'']$,  we have
\[ D/2\cong Im(i^*/2: CH^*(BG_k)\to CH^*(BB_k)/2).\]
\end{thm}

We can see the map $i^*$ is given 
\[ c_4\mapsto c_2^2,\ \  c_6\mapsto c_3^2,\ \
 c_8\mapsto e_8, \ \  c_{16}''\mapsto (c_4)^4, \] 
\[ c_2''\mapsto 2c_2,\ \ c_4''\mapsto 2c_4,\ \
c_6''\mapsto 2c_2c_4,\ \ 
 c_8''\mapsto 2c_4^2,\ \ c_{10}''\mapsto 2c_2c_4^2,\ \
\] \[
c_{12}''\mapsto 2c_4e_8, \ \ c_{14}''\mapsto
2c_2c_4e_8.\]
Here $c_i''=c_i(\Delta_{\bC})$ for the complex spin representation.  
 
From Theorem 10.2, we have 
\begin{prop}  Let $G=Spin(9)$ and $\bG$ be versal,
Then we have
 \[\tilde D_{CH/2}(\bG)= \tilde D(\bG)\cong (\bZ/2\{1,
c_2c_3\}\otimes \bZ/2[c_4]/(c_4^4))^+\otimes
S(t,c).\]
 %\[ \tilde D_{CH}(\bG)\cong ((\bZ_{(2)}\{1\}\oplus \bZ/2\{c_2c_3\})\otimes \bZ_{(2)}[c_4]/(2c_4^2,c_4^4))^+ .\]
\end{prop}

\section{The ordinary cohomology for $F_4$　}

In this and next sections, we assume $(G,p)=(F_4,3)$.
For ease of notation,  the classifying space 
$BG$ means the topological space $BG(\bC)$
(or the variety $BG_{\bar k}$).
.
Toda computed the $mod(3)$ cohomology of $BF_4$.
(For details see \cite{Tod1}.)
\begin{thm} (Toda [Tod1])
We have additively \  $ H^*(BG;\bZ/3)\cong C\otimes D,$
\[where \quad    C=F\{1,x_{20},x_{20}^2\}+\bZ/3[x_{26}]\otimes \Lambda(x_9)\otimes \bZ/3 \{1,x_{20},x_{21},x_{26}\}\]
\[ and\ \ \   D=\bZ_{(3)}[x_{36},x_{48}],\quad F=\bZ_{(3)}[x_4,x_8]. \]
Here  the suffix means its degree.
\end{thm}

{\bf Remark.}
The multiplicative structure is also given completely
by Toda \cite{Tod1}, e.g., $x_{21}x_8+x_{20}x_9=0$.

Note  that $H^*(BG)$ has no higher $3$-torsion and
$Q_0x_8=x_9$, $Q_0x_{20}=x_{21}$.
So $x_8,x_{20}\not \in H^*(BG)$.
From $Q_0x_{25}=x_{26}$, we can see
 $x_{26}=Q_2Q_1(x_4)$.  Using these we have 
\begin{cor} ([Tod1], [Ka-Mi])
We have isomorphisms
\[H^*(BT;Z/3)^W\cong H^{even}(BG;\bZ/3)/(Q_2Q_1x_4)
\cong D/3\otimes F\{1,x_{20},x_{20}^2\}.\]
\[ H^*(BT)^W\cong H^*(BG)/Tor\cong D\otimes(\bZ_{(3)}\{1,x_4\}\oplus E)\]
where $D=\bZ_{(3)}[x_{36},x_{48}]$, $F=\bZ_{(3)}[x_4,x_8]$,
and $E=F\{ab|a,b\in \{x_4,x_8,x_{20}\}\}$.
\end{cor}
Note that
$E\oplus\bZ_{(3)}\{1,x_4,x_8,x_{20}\} \cong \bZ_{(3)}[x_4,x_8,x_{20}]/(x_{20}^3)$.

To show the above theorem, Toda uses the following 
fibering 
\[\Pi \to BSpin(9)\to BF_4\]
where $\Pi=F_4/Spin(9)$ is the Cayley plane.
Let $T$ be the maximal torus of $Spin(9)\subset F_4$,
and $W(G)$ be the Weyl group of $G$.
Let us write $H^*(BT;\bZ/3)\cong
\bZ/3[t_1,...,t_4]$.  It is well known
\[ H^*(BSpin(9);\bZ/3)\cong H^*(BT;\bZ/3)^{W(Spin(9))}\cong
\bZ/3[p_1,...,p_4]\]
where $p_i$ is the $i$-th Pontrjagin class 
which is the $i$-th elementary symmetric function on variable $t_j^2$.  The Weyl group $W(F_4)$
is generated by elements in $W(Spin(9))$
and by $R$ with $R(u_i)=u_i-(u_1+...+u_4)$.  The invariant ring of
$G=F_4$ is also computed by Toda
\begin{thm}  There is a ring isomorphism
\[ H^*(BT;\bZ/3)^{W(G)}
\cong 
\bZ/3[p_1,\bar p_2,\bar p_5,
\bar p_9,\bar p_{12}]/(r_{15})
\subset
 \bZ/3[p_1,...,p_4]\]
\[ where \quad 
\bar p_2=p_2-p_1^2,\ \ \bar p_5=p_4p_1+p_3\bar p_2,\ \ 
\bar p_9=p_3^3\ mod(I), \]
\[\bar p_{12}=p_4^3\ mod(I),\ \ r_{15}=\bar p_5^3,\ \ with \ I=Ideal(p_1,\bar p_2).\]
\end{thm}
Let us write $i: T\subset F_4$.  The above elements
correspond even degree generator (except for $x_{26}$).
\begin{cor}
We have
\[ i^*(x_4)=p_1,\ \ i^*(x_8)=\bar p_2,\ \  i^*(x_{20})=\bar p_5,\ \  i^*(x_{36})=\bar p_9,\ \  i^*(x_{48})=\bar p_{12}.\]
\end{cor}
By using this corollary, we can write the reduced power
actions.
\begin{lemma} (\cite{Tod1})  We have
\[ P^1(x_4)=-x_8+x_1^2,\ \ P^1(x_8)=x_4x_8,\ \
P^1(x_{20})=0,\]
\[ P^3(x_4)=0, \ \  P^3(x_8)=x_{20}-x_4x_8^2,\ \
P^3(x_{20})=x_{20}x_4(-x_8+x_4^2),\]
\[P^3(x_{36})=x_{48}\ mod(x_4,x_8).\]
\end{lemma}

Recall that the  $mod(3)$ cohomology of $F_4$ is 
\[ H^*(G;\bZ/3)\cong \bZ/3[y_8]/(y_8^3)\otimes
\Lambda(x_3,x_7,x_{11},x_{15}).\]
Here suffices mean their degree.
Recall the cohomology of the flag variety 
\[H^*(G/T;\bZ/3)\cong P(y)\otimes S(t)/(b_1,...,b_4)\]
and so $b_1=p_1,b_2=\bar p_2,b_3=p_3,b_4=p_4$.
Define   $D_{H/3}(G)=Ker(j^+)/(Im(i^+)$ for
 \[  H^*(BG;\bZ/3)\stackrel{i^*}{\to}
    H^*(BT;\bZ/3)\stackrel{j^*}{\to} H^*(G/T;\bZ/3).\]
\begin{prop} We have additively 
\[D_{H/3}(G)\cong \bZ/3[p_3,p_4]^+/(p_3^3,p_4^3)
\otimes S(t,p)\quad for \ S(t,p)\cong S(t)/(p_1,...,p_4).\]
\end{prop}
\begin{proof}
First note that $i^*(x_4)=p_1$, $i^*(x_8)=\bar p_2$ and 
$p_1,\bar p_2$ are zero in $D_{H/3}(G)$.

Since $i^*(x_{36})=\bar p_9=p_3^3$ $mod(I)$, we see $p_3^3=0\in \bar D_{H/3}(G)$.
Similarly, we see $p_4^3=0\in D_{H/3}(G)$
from $i^*(x_{48})=\bar p_{12}$.
\end{proof}

\section{$BP^*$-theory and Chow ring for $(F_4,3)$}

We consider the Atiyah-Hirzebruch spectral
sequence \cite{Ko-Ya}
\[ E_2^{*,*'}\cong H^*(BG)\otimes BP^*
\Longrightarrow BP^*(BG).\]
Its differentials have forms of 
$ d_{2p^n-1}(x)=v_n\otimes Q_n(x).$
Using $Q_1(x_4)=x_9, Q_1(x_{20})=x_{25}, Q_1(x_{21})=x_{26}$ and $Q_2x_{9}=x_{26}$, we can compute (\cite{Ko-Ya})
\[ E_{\infty}^{*,*'}
\cong D\otimes(BP^*\otimes(\bZ_{(3)}\{1,3x_4\}\oplus E)
 \oplus BP^*/(3,v_1,v_2)[x_{26}]^+).\]
Hence we have
\begin{thm}  (\cite{Ko-Ya}, \cite{YaF4})  We have the isomorphism
\[BP^*(BG)\otimes_{BP^*}\bZ_{(3)}\cong D\otimes
    (\bZ_{(3)}\{1,3x_4\}\oplus E\oplus \bZ/3[x_{26}]^+).\]
\end{thm}

\begin{lemma}  (\cite{YaF4}) We see $x_{26}\in Im(cl)$.
\end{lemma}
\begin{proof}
From Lemma 4.3 in \cite{YaF4}, (see also \cite{Ka-Ya})  if $x\in H^4(X(\bC)$
and $px\in Im(cl)$, then there is $x'\in H^{4,3}(X;\bZ/p)$
such that $cl(x')=x\ mod(p)$. Note
\[ y=Q_2Q_1(x')\in H^{26.13}(X:\bZ/3)\cong CH^{13}(X)/3.\]
Hence we have the lemma from $x_{26}=cl(y)$.
\end{proof}

Let  $Grif \subset Tor\subset CH^*(X|_{\bar k})$ be the ideal generated by Griffiths elements i.e., $Grif=Ker(t_{\bC})$ for $t_{\bC}
: CH^*(X|_{\bar k})\to H^*(X)$
\begin{cor}  We have  $Tor/Grif \cong D\otimes \bZ/3[x_{26}]\{x_{26}\}$ and
\[ CH^*(BG_{\bar k})/Tor\subset   D\otimes (\bZ_{(3)}\{1,3x_4\}\oplus E)\subset H^*(BG)/Tor.\]
\end{cor}

If Totaro's conjecture is correct, then $Grif=\{0\}$
and the first inclusion is an isomorphism.

From Lemma 3.1-3.4 in \cite{YaF4}, we see
$x_{36},3x_4,x_4^3,...$ are represented by Chern classes.
Moreover we still know
\begin{lemma} (\cite{YaF4})
Let $RP$ be the subalgebra of
the $mod(3)$ Steenrod algebra $A_3$ generated by reduced powers.
Then $(BP^*(BG)\otimes _{BP^*}\bZ_{(3)})/(Tor,3)$ is generated 
as an $PR$-module by
\[ x_4^2,\ \ x_8^2,\ \ and \ \ 
 products\ of\ some\ Chern\ classes.\]
\end{lemma}

Here we consider the (algebraic) $K$-theory
with the coefficient $K^*=\bZ_{(p)}[v_1,v_1^{-1}]$
such that
\[ BP^*(BG)\otimes _{BP^*}K^*\cong K^*(BG).\] 
Recall that  $gr_{geo}^*(X)$ is the graded associated ring
defined by the geometric filtration of $K^0(X)$ 
(that is isomorphic to the infinite term $E_{\infty}^{2*,*,0}$
of the motivic Atiyah-Hirzebruch spectral sequence). 
Then it is well known that we have the surjection
$ CH^*(X)\to  gr_{geo}^*(X).$

\begin{lemma}  We see $x_4^2\in Im(cl).$
\end{lemma}
\begin{proof}
Suppose that $x_4^2\not \in CH^*(BG_k)$.  However
$x_4^2$ exists in $K^*(BG_k)\cong K^*(BG)$, because
it exists in $BP^*(BG)$.
Since 
$ CH^*(X)\to gr_{geo}^*(X)$ is surjective,  there is an
element 
\[ c\in CH^*(BG_k)\quad such\ that \ \
    c=v_1^sx_4^2\ for\ s\ge 1.\]By dimensional reason,
this $s=1$ and $|c|=4$.  But by Totaro
\[CH^2(BG)\cong (BP^*(BG)\otimes _{BP^*}\bZ_{(p)})^4,\]
which is a cotradiction.
\end{proof}

\begin{prop}  ([Ya1])
Let $(G,p)=(F_4,3)$.
Suppose $x_8^2\in Im(cl)$.  Then
the modified cycle map  
$\bar cl :CH^*(BG_k)\to BP^*(BG)\otimes_{BP^*}\bZ_{(3)}$ 
 is surjective.   Moreover, we  have
\[ Im(\bar cl)\cong Im(cl)\cong D\otimes (
\bZ_{(3)}\{1,3x_4\}\oplus E
\oplus \bZ/3[x_{26}]^+).\]
(Here $D=\bZ_{(3)}[x_{36},x_{48}]$ and
$E\oplus\bZ_{(3)}\{1,x_4,x_8,x_{20}\}=\bZ_{(3)}[x_4,x_8,x_{20}]/(x_{20}^3))$.
\end{prop}

From Theorem 2.3, we have
\[ CH^*(\bG/B_k)/3\cong S(t)/(p_ip_j|1\le i,j\le 4).\]
Hence, we have 
$ (p_ip_j)\supset Ideal(i^*CH^*(BG_k))$ 
e.g. $i^*(x_4^2)=p_1^2$, $i^*(x_4x_8)=p_1p_2$,...

Suppose that $x_8^2\not \in CH^*(BG_k)$.  However
$x_8^2$ exists in $K^*(BG_k)\cong K^*(BG)$, because
it exists in $BP^*(BG)$.
Since 
$ CH^*(X)\to gr_{geo}^*(X)$ is surjective,  there is an
element 
\[ c\in CH^*(BG_k)\quad such\ that \ \
    c=v_1^sx_8^2\ for\ s\ge 1.\]
This $c$ is torsion element in $CH^*(BG)$ since
$3x_8^2\in Im(cl)$.
\begin{prop}
If $x_8^2\not \in Im(cl)$, then there is a non zero element $c\in Tor$ with $|c|=16-4s$ for $s=1$ or $2$.
\end{prop}

 We consider
the following ideals in $CH^*(BB_k)$
\[ Ker(j^*)=(3p_1, p_1^2, p_1\bar p_2, 3p_3,\bar p_2^2,....)\supset (3x_4,x_4^2, x_4x_8, x_4^3, \lambda x_8^2,..)=Ideal(Im(i^*)),\]
for $\lambda\in \bZ_{(3)}.$
We note that  
\[i^*(3x_4)=3p_1,\ \ i^*(x_4^2)=p_1^2,\ \ 
i^*(x_4^3)=3p_3,\ \ i^*(x_4x_8)=p_1p_2\quad \]
where we used $p^3_1=3p_3\ mod(p_1p_2)$.
Note that $\lambda\not =0$ implies $i^*(x_8^2)=p_2^2$.
\begin{prop}
The map $\tilde cl$ is surjective if and only if
$D^*(\bG)
=0$ for $*\le 16$.
\end{prop}

\begin{prop} The ring $\tilde D(\bG)$ is  isomorphic to a quotient of
\[ D(F_4)'=\bZ/3\{p_1^{i_i}p_2^{i_2}p_3^{i_3}p_4^{i_4}|
2\le i_1+...+i_4\}/(
p_1^2,p_1p_2,p_3^3,p_4^3).\]
\end{prop}
%\begin{proof} The module $D(\bG)$ is a quotient of
%the above right hand module in the lemma.
%The fact $p_3^2p_4^2\not =0$ follows from 
%Lemma 6.6.
%\end{proof}

\section{ $E_6$,$E_7$ for $p=3$}

The groups $E_6$, $E_7$ for $p=3$ are of type $(I)$.
Hence 
\[Kerj^+(\bG)\cong Ideal(b_ib_j, b_k|1\le i,j\le 4,\ 5\le k\le \ell)\subset S(t)/3.\]

By Kameko [Ka], there is a representation 
$\rho_{\ell}: E_{\ell}\to U(N)$ such that 
\[ i^*c_{18}(\rho_{\ell})=x_{36}\quad for\ i_{\ell}:F_4\to E_{\ell}.\]
Hence $i_{\ell}^*(P^3c_{18})=x_{48}$.  Thus 
\[p_3^3=i^*(c_{18}),\quad p_4^3=i^*(P^3c_{18}).\]

\begin{prop}  Let $G=E_{\ell}$ for $\ell=6$ or $7$.
Then there is a surjection
\[ ((\bZ/3\{1\}\otimes D(F_4)')\otimes
\bZ/3[b_5,...,b_{\ell}])^+ \to
 \tilde D(\bG).\]
\end{prop}
\begin{proof}
From  the proof of Lemma 12.5, we see $p_1^2\in Im(i^*)$.
Since  $P^1(p_1^2)=p_1\bar p_2$, we see $p_1p_2\in Im(i^*)$ also for $E_6,E_7$.
\end{proof}


\begin{thebibliography}{R-W-Y}



\bibitem[Be-Wo]{Be-Wo}
D. Benson and  J. Wood
\newblock  Integral invariants and cohomology of $BSpin(n)$.
\newblock \emph{Topology.}
\textbf{34} (1994),  13-28.

\bibitem[Fe]{Fe}
M. Feshbach.
\newblock  The image of $H^*(BG;\bZ)$ in $H^*(BT;\bZ)$ for a compact Lie groupwith maximal torus $T$.
\newblock \emph{Topology.}
\textbf{20} (1985),  93-95.








\bibitem[Ga-Me-Se]{Ga-Me-Se}
S. Garibaldi, A. Merkurjev and J-P Serre.
\newblock Cohomological invariants in Galois cohomology.
\newblock \emph{University Lecture Series 28,
Amer. Math. Soc. Providence, RI}
(2003), viii+168pp.






\bibitem[Gu]{Gu}
P. Guillot,
\newblock  The Chow rings of $G_2$ and $Spin(7)$,
\newblock \emph{J. reine angew. Math.}
\textbf{604} (2007),  137-158.




\bibitem[Ha]{Ha}
M.Hazewinkel,
\newblock  Formal groups and applications,
\newblock \emph{Pure and Applied Math. 78, Academic Press Inc.} 
\textbf{} (1978), xxii+573pp. 

%\bibitem[Ho]{Hodkin}
%L. Hodgkin.
%\newblock  The equivariant K\"unneth theorem in K-theory,
%Topics in K-theory, Two independent contributions.
%Lect. Note. Springer, Berlin,
%{\bf 496} (1975), 1-101.

\bibitem[Ka]{Ka}
M. Kameko..
\newblock  Mod (3) Chern classes and generators.
\newblock \emph{Proc. Japan Acad. } 
\textbf{93} (2017), 55-60.




%\bibitem{Ka-Mi}
%M. Kameko and M. Mimura,
%\newblock Wely group invariants,
%\newblock {arXiiv: 1202.6459v1 [math.AT]},
 %(2012).



\bibitem[Ka-Ya]{Ka-Ya}
M. Kameko and N. Yagita.
\newblock  Chern subrings.
\newblock \emph{Proc. Amer. Math. Soc.} 
\textbf{138} (2010), 367-373.





%\bibitem[Ka-Me]{Ka-Me} 
%N.Karpenko and A. Merkurjev.
%\newblock  On standard norm variety.
%\newblock \emph{Ann. Sci. Ec. Norm. Super.}
% \textbf{46} (2013), 214.

\bibitem[Kar]{Kar} 
N. Karpenko.
\newblock  Chow groups of some generically twisted 
flag varieties.
\newblock \emph{Ann. K-theory}
 \textbf{2} (2017), 341-356.



%\bibitem{Ka2} 
%N. Karpenko.
%\newblock  On generic flag varieties of Spin(11) and Spin(12)..
%\newblock \emph{http//sites.ualberta.ca/~karpenko
%/publ/spin11pdf}
% \textbf{ } (2017).






%\bibitem{Ka-Me}
%N. Karpenko and Merkurjev.
%\newblock \emph{On standard norm varieties.}
%Ann. Sci. Ec. Norm. Super. (4)
%{\bf 46} (2013), 175-214.



\bibitem[Ko]{Ko}
A. Kono,
\newblock On the integral cohomology of $BSpin(n)$.
\newblock \emph{J. Math. Kyoto Univ.}
\textbf{26} (1986),  333-337.




\bibitem[Ko-Ya]{Ko-Ya}
A. Kono and N.Yagita,
\newblock Brown-Peterson and ordinary cohomology theories of classifying spaces for compact Lie groups,
\newblock \emph{Trans. Amer. Math. Soc.}
\textbf{339} (1993),  781-798.





%\bibitem[Ka-Mi]{Ko-Mi} 
%A. Kono and M.~Mimura.
%\newblock  Cohomology operations and the Hopf algebra structures of compact exceptional Lie groups $E_7$ and $E_8$.
%\newblock \emph{Proc London Math. Soc.} 
%\textbf{35}  (1977), 345-358.





\bibitem[Me-Ne-Za]{Me-Ne-Za}
A. Merkurjev, A. Neshitov  and K. Zainoulline.
\newblock \emph{Invariants o degree $3$ and torsion in the Chow group of a versal flag.}
Composito Math. {\bf 151} (2015), 11416-1432.






%\bibitem[Me-Su]{Me-Su}
%A. Merkurjev and A. Suslin.
%\newblock \emph{Motivic cohomology of the simplicial
%motive of a Rost variety.}
%J. Pure and Appl. Algebra {\bf 214} (2010), 2017-2026.








\bibitem[Mi-Ni]{Mi-Ni} 
M.~Mimura and T.~Nishimoto.
\newblock Hopf algebra structure of Morava K-theory of exceptional Lie 
groups.
\newblock \emph{Contem. Math.} 
\textbf{293}  (2002), 195-231.



%\bibitem[Mi-Tod]{Mi-Tod} 
%M. Mimura and H. Toda, {\it Topology of Lie groups}, I and II, Translations of Math. Monographs, Amer. Math. Soc, {\bf 91} (1991).


%\bibitem{Na} 
%M. Nakagawa.
%\newblock The integral cohomology ring of $E_8/T$.
%\newblock \emph{Proc. Japan } \textbf{86} (20), 383-394.


\bibitem[Ni]{Ni} 
T.~Nishimoto.
\newblock Higher torsion in Morava K-thoeory of $SO(m)$ and $Spin(m)$.
\newblock \emph{J. Math. Soc. Japan. } \textbf{52} (2001), 383-394.

% \bibitem[Or-Vi-Vo]{Or-Vi-Vo} 
%D.Orlov,A.Vishik and V.Voevodsky.
%\newblock An exact sequence for Milnor's K-theory with
% applications to quadric forms.
%\newblock \emph{Ann. of Math.} \textbf{165} (2007) 1-13.


% \bibitem[Or-Vi-Vo]{Or-Vi-Vo} 
%D.Orlov,A.Vishik and V.Voevodsky.
%\newblock An exact sequence for Milnor's K-theory with
% appl



%\bibitem{Pa}  
%I. Panin.   
%\newblock On th algebraic $K$-theory of twisted flag varieties.
%\newblock \emph{K-theory}
%{\bf 8}
%(1994), 541-585.


\bibitem[Pe]{Pe}  
V. Petrov.   
\newblock Chow ring of generic maximal orthogonal Grassmannians. 
\newblock \emph{Zap. Nauchn, Sem. S.-Peterburg. Otdel. Mat. Inst. Skelov. (POMI) 443}
(2016), 147-150.




%\bibitem{Pe-Se}  
%V. Petrov and N. Semenov.  
%\newblock Rost motives, affine varieties, and classifying spaces. 
%\newblock \emph{arXiv : 1506.007788v1.}
%(math. AT) Jun 2015.



\bibitem[Pe-Se-Za]{Pe-Se-Za}  
V.Petrov, N.Semenov and K.Zainoulline.  
\newblock J-Invariant of linear algebraic groups.
\newblock \emph{Ann. Scient. Ec. Norm. Sup.}
{\bf 41}, (2008) 1023-1053.




\bibitem[Qu]{Qu}
D. Quillen, 
\newblock The mod $2$ cohomology rings of extra-special $2$-groups and the spinor groups,
\newblock \emph{Math. Ann.}
\textbf{194} (1971), 197-212.








\bibitem[Ra]{Ra}
D.Ravenel.
\newblock Complex cobordism and stable homotopy groups of spheres. 
\newblock \emph{Pure and Applied Mathematics, 121. Academic Press}
(1986).



\bibitem[Ro]{Ro}  
M.Rost. 
\newblock Some new results on Chowgroups of quadrics.
\newblock \emph{preprint} (1990).



%\bibitem[Ro]{Ro}  
%M.Rost. 
%\newblock On the basic correspondence of a splitting variety.
%\newblock \emph{preprint} (2006)


\bibitem[Sc-Ya]{Sc-Ya}
B. Schuster\ and\ N. Yagita, Transfers of Chern classes in BP-cohomology and Chow rings, 
Trans. Amer. Math. Soc. {\bf 353} (2001),   1039-1054.






%\bibitem{Se} 
%N. Semenov. 
%\newblock Motivic construction of cohomological invariant.
%\newblock \emph{Comment. Math. Helv.}
%\emph{91} (2016) 163-202.




\bibitem[Se-Zh]{Se-Zh} 
N. Semenov and M Zhykhovich. 
\newblock Integral motives, relative
Krull-Schumidt principle, and Maranda-type theorems.
\newblock \emph{Math. Ann.} 
\emph{363} (2015) 61-75.



%\bibitem[Sm-Vi]{Ro2}  
%A. Smirnov and A. Vishik. 
%\newblock Subtle characteristic classes.
%\newblock \emph{arXiv : 1401.6661v1 [math.AG]}
%(2014).





%\bibitem[Su-Jo]{Su} 
%A.Suslin and S.Joukhovistski. 
%\newblock Norm Variety.
%\newblock \emph{J.Pure and Appl. Algebra} 
%\emph{206} (2006) 245-276.



 



\bibitem[Tod1]{Tod1} 
H. Toda,  
\newblock Cohomology $mod(3)$ of the classifying space $BF_4$ of the exceptional group $F_4$.
\newblock \emph{J. Math. Kyoto Univ.}   \textbf{13} 
(1973) 97-115.





\bibitem[Tod2]{Tod2}
H.Toda.
\newblock On the cohomolgy ring of some homogeneous spaces.
\newblock \emph{J. Math. Kyoto Univ.} \textbf{15} (1975), 185-199.





\bibitem[Tod-Wa]{Tod-Wa}
H.Toda and T.Watanabe.
\newblock The integral cohomology ring of
$F_4/T$ and $E_4/T$.
\newblock \emph{J. Math. Kyoto Univ.} \textbf{14} (1974), 257-286.




\bibitem[To1]{To1} 
B.~Totaro.  
\newblock The Chow ring of classifying spaces. 
\newblock \emph{ Proc.of Symposia in Pure Math. "Algebraic
K-theory" 
(1997:University of Washington,Seattle)} \textbf{67} (1999), 248-281.


%\bibitem{To1} 
%B.~Totaro.  
%\newblock The torsion index of $E_8$ and the other groups.
%\newblock \emph{Duke Math. J.} \textbf{299} (2005), 219-248.




\bibitem[To2]{To2} 
B.~Totaro.  
\newblock The torsion index of the spin groups.
\newblock \emph{Duke Math. J.} \textbf{299} (2005), 249-290.


\bibitem[To3]{To3} 
B.~Totaro,  
\newblock Group cohomology and algebraic cycles, 
\newblock \emph{Cambridge tracts in Math.  (Cambridge Univ. Press)} \textbf{204} (2014).







%\bibitem[Vi]{Vi}
%A. Vishik.
%On the Chow groups of quadratic Grassmannians.
%\newblock \emph{Doc. Math.}
%
%{\bf 10} (2005), 111-130.


%\bibitem{Vi-Ya}  
%A.Vishik and N.Yagita. 
%\newblock Algebraic cobordisms of a Pfister quadric.
%\newblock \emph{ London Math. Soc.}
%{\bf 76} (2007), 586-604.



% \bibitem[Vi-Za]{Vi-Za}  
%A.Vishik and K.Zainoulline. 
%\newblock Motivic splitting lemma.
%\newblock \emph{ Doc. Math.}
%{\bf 13} (2008), 81-96.








%\bibitem[vo1]{Vo1} 
%V.~Voevodsky, 
%\newblock The Milnor conjecture,
%\newblock \emph{www.math.uiuc.edu/K-theory/0170}    (1996).




\bibitem[Vo1]{Vo1} 
V.~Voevodsky. 
\newblock Motivic cohomology with $\bZ/2$ coefficient.
\newblock \emph{Publ. Math. IHES}
\textbf{98} (2003), 59-104.

%\bibitem[Vo2]{Vo2}
%V.~Voevodsky (Noted by Weibel).
%\newblock Voevodsky's Seattle lectures : $K$-theory and %motivic cohomology
%\newblock \emph{Proc.of Symposia in Pure Math. "Algebraic %K-theory" 
%(1997:University of Washington,Seattle)} \textbf{67} (1999), %283-303.

%\bibitem[Vo3]{Vo3} 
%V.Voevodsky. 
%\newblock Reduced power operations in motivic cohomology.
%\newblock \emph{Publ.Math. IHES}
%\textbf{98} (2003),1-57.


\bibitem[Vo2]{Vo2} 
V.Voevodsky. 
\newblock On motivic cohomology with $\bZ/l$-coefficients.
\newblock \emph{Ann. of Math.}
{\bf 174}    (2011), 401-438.


\bibitem[Ya1]{YaCo}
N.~Yagita,
\newblock Chow rings  of classifying spaces of 
exraspecial $p$-groups.
\newblock \emph{Recent progress in homotopy theory. Contempt. Math.}
\textbf{293} (2002), 397-409..



\bibitem[Ya2]{YaF4}
N.~Yagita,
\newblock The image of the cycle map of classifying
space of the exceptional Lie group $F_4$,
\newblock \emph{J. Math. Kyoto Univ.}
\textbf{44} (2004), 181-191.



%\bibitem[Ya2]{Ya2}
%N.~Yagita.
%\newblock  Applications of Atiyah-Hirzebruch spectral 
%sequence for motivic cobordism.
%\newblock \emph{Proc. London Math. Soc.}
%\textbf{90} (2005) 783-816.

%\bibitem[Ya3]{Ya3}
%N.~Yagita.
%\newblock  Chow rings of excellent quadrics.
%\newblock \emph{J. Pure and Appl. Algebra}
%\textbf{212} (2008), 2440-2449. 



%\bibitem[Ya4]{Ya4}
%N.~Yagita.
%\newblock{ Motivic cohomology of quadrics
%and coniveau spectral sequence.}
%\newblock \emph{J. K-theory}
%\textbf{6}  (2010), 547-589.





%\bibitem[Ya4]{Ya4}
%N.~Yagita.
%\newblock Note on Chow rings of nontrivial $G$-torsors over a field.
%\newblock \emph{Kodai Math. J.}
%{\bf 34} (2011), 446-463.

 
%\bibitem{Ya3}
%N.~Yagita.
%\newblock  Algebraic $BP$-theory and norm varieties.
%\newblock \emph{Hokkaido Math. J.}
%{\bf 41} (2012), 275-316.

%\bibitem[Ya4]{Ya4}
%N.~Yagita.
%\newblock  Algebraic cobordism and flag varieties.
%\newblock \emph{arXiv:1510.08135 (math. AT).}
%(2016)



%\bibitem{YaF}
%N.~Yagita.
%\newblock  Note on the filtrations of the $K$-theory.
%\newblock \emph{Kodai Math. J.}
%{\bf 38} (2015), 172-200.


\bibitem[Ya3]{YaW}
N.~Yagita.
\newblock  Note on restriction maps of Chow rings to
Weyl group invariants.
\newblock \emph{Kodai Math. J.}
{\bf 40} (2017), 537-552.


\bibitem[Ya4]{YaC}
N.~Yagita.
\newblock  Chow rings of versal complete flag varities.
\newblock \emph{arXiv:1609.08721vl. (math. KT).}
(2016)


%\bibitem{Za}
%K. Zainoulline,
%\newblock  Twisted gamma filtration of a linear algebraic group,
%\newblock \emph{Compositio Math.}
%\textbf{148} (2012), 1645-1654.













\end{thebibliography}
\end{document}